\documentclass[12pt]{article}
\usepackage{graphics,latexsym,amssymb,amsmath,amscd}
\RequirePackage{graphics}
\usepackage{multirow}
\usepackage[all,cmtip]{xy}
\usepackage{xcolor}
\def\[#1\]{\begin{eqnarray*}#1\end{eqnarray*}}

\def\tr{\hbox{\rm tr}\,}

\def\SL{\hbox{\bf SL}}
\def\phi{\varphi}

\input xy
\xyoption{all}
\newtheorem{thm}{Theorem}[section]
\newtheorem{dfn}[thm]{Definition}

\newtheorem{cor}[thm]{Corollary}
\newtheorem{prop}[thm]{Proposition}
\newtheorem{lemma}[thm]{Lemma}

\newcommand{\Pf}{{\em Proof}. }
\newcommand{\EPf}{\hbox{}\hfill$\Box$\vspace{.5cm}}
\newcommand{\C}{{{\mathbb C}}}

\newcommand{\R}{{{\mathbb R}}}
\newcommand{\Z}{{{\mathbb Z}}}

\def\tr{\hbox{\rm tr}\,}

\def\SL{\hbox{\bf SL}}

\def\phi{\varphi}

\newenvironment{proof}{{\normalsize {\sc Proof}:}}{{{\hfill $\Box$\vspace{.5cm}}}}

\def\C{{\mathbb C}}
\def\R{{\mathbb R}}

\date{}
\title{Reductions of path structures and classification of homogeneous structures in dimension three}
\author{ E. Falbel (corresponding author), Martin Mion-Mouton and Jose Miguel Veloso}
 
\begin{document}

\maketitle
\newcommand{\D}{\mbox{$\cal D$}}
\newtheorem{df}{Definition}[section]
\newtheorem{te}{Theorem}[section]
\newtheorem{co}{Corollary}[section]
\newtheorem{po}{Proposition}[section]
\newtheorem{lem}{Lemma}[section]
\newcommand{\Ad}{\mbox{Ad}}
\newcommand{\ad}{\mbox{ad}}
\newcommand{\im}[1]{\mbox{\rm im\,$#1$}}
\newcommand{\bm}[1]{\mbox{\boldmath $#1$}}
\newcommand{\sime}{\mbox{sim}}
\begin{abstract}
In this paper we show that if a path structure has non-vanishing curvature at a point then it has a canonical reduction to a $\Z/2\Z$-structure at a neighbourhood of that point (in many cases it has a canonical parallelism).   A simple implication of this result is that
the automorphism group of a non-flat path structure is of maximal dimension three (a result by Tresse of 1896).  We also classify the invariant path structures on three-dimensional Lie groups.

Keywords: Path structures, Homogeneous spaces, Cartan connections, automorphism group.
\end{abstract}

\section{Introduction}

A path structure on a 3-manifold is a choice  of two subbundles $T^1$ and $T^2$ in $TM$ such that
$T^1 \cap T^2=\{ 0\}$ and such that $T^1 \oplus T^2$ is a contact distribution.  This geometry gives rise to a Cartan connection on a canonical principal  bundle (with structure group, $B_\R$, the Borel subgroup of upper triangular matrices in $\SL(3,\R)$) which we call $Y$ (\cite{Car}, see \cite{IL} for a modern presentation and section \ref{section:Cartan}).  There are two curvature functions $Q^1$ and $Q^2$ defined on $Y$ which should determine, in certain situations, the path structure up to equivalence.  Indeed, when $Q^1=Q^2=0$ the path structure is locally equivalent to the path structure on the model space $
\SL(3,\R)/B_\R
$ (see section \ref{section:Cartan}).

A simple way to define a path structure on a 3-manifold is to fix a contact form and two transverse vector fields contained in the kernel of the form.    In particular, this defines a parallelism of  the 3-manifold.  Reciprocally, one might ask whether there exists a canonical parallelism, with transverse vector fields contained in the contact distribution, 
associated to a path structure.  In this case, the automorphism group of the path structure should coincide with the automorphism group of the parallelism.

We show in this paper that if a path structure has non-vanishing curvature at a point then it has a canonical reduction to a $\Z/2\Z$-structure at a neighbourhood of that point (in many cases it has a canonical parallelism).

In section \ref{section:Cartan} we recall the construction of the Cartan bundle $Y$ and adapted connection to a path structure on a 3-manifold.   The curvature invariants $Q^1$ and $Q^2$ are related to the invariants found by Tresse (see the explicit computation in section 3.5 in \cite{FVODE}).
We also recall the definition of strict path structure, that is, when a contact form is fixed over the manifold and the definition  of a Cartan bundle $Y^1$ and  connection adapted to that structure which was used in (\cite{FMMV}) to obtain a classification of compact 3-manifolds with non-compact automorphism group preserving the strict path structure. We also give details in appendix \ref{section:pseudo-embedding} for a natural embedding $Y^1\to Y$ which is used to compute curvatures of the homogeneous path structures in section \ref{section:homogeneous}.

In section \ref{section:reduction} we prove a canonical reduction of a path structure when the structure is non-flat:
\begin{thm}\label{theorem-intro}
If the path structure is not flat, there exists a canonical reduction of the fiber bundle $Y$ to a $\Z/2\Z$-structure.  \end{thm}
A more precise theorem is proved in section \ref{section:reduction} where we give conditions for the existence of a further reduction to  a parallelism.  The theorem implies the classical theorem by Tresse that a non-flat path structure has an automorphism group of dimension at most three (\cite{tresse}).
While locally homogeneous path structures
of course realize this upper bound,
Tresse's theorem does not, however, answer the question
wether a path structure having a three-dimensional
automorphism group, is locally homogeneous or not.
A one-dimensional subgroup of automorphisms may
indeed fix a point, forbidding the orbit of the automorphism
group at this point to be open.
A direct and important Corollary of Theorem \ref{theorem-intro},
proved in section \ref{subsection:3dim},
shows that this phenomenon does not happen.
\begin{cor}\label{corollaire-intro}
At a point where the curvature is non-zero and
the algebra of local Killing fields
has dimension at least three,
a path structure is locally isomorphic to a
left-invariant path structure on a three-dimensional
Lie group.
\end{cor}
This is a motivation to classify
left invariant path structures on three dimensional Lie groups,
which is done
in section  \ref{section:homogeneous}.  The results are gathered in tables in section \ref{subsection:table}.  We choose for each  structure a parallelism (some are canonical) and we compute the curvatures for each of these invariant structures using an embedding of the group in the corresponding Cartan bundle $Y$ (see Proposition \ref{proposition:curvatures} in section \ref{section:curvatures}).

In the last section \ref{subsection-invariantsSL2} we give a  geometric description of the invariant structures on $\SL(2,\R)$ involving the type of the contact plane with respect to the Killing metric and a cross-ratio which parametrizes the positions of the one dimensional distributions in the contact plane.   Similar descriptions can be made for each of the three dimensional groups. 

It is interesting to note that homogeneous CR structures in dimension three were classified by Cartan in \cite{Car-CR} (see a presentation analogous to the classification of path structures in the present paper in \cite{FG}).  We did not find, however,  the analogous classification of path structures in his work.  The gap phenomenon which is a consequence
of theorem \ref{theorem-intro} is more general (see \cite{gapphenomenon}) and, in particular, is valid for CR structures.  But the detailed reductions obtained in section  \ref{section:reduction} and the corollary above were not worked out in the case of CR structures as far as we know.

\section{The Cartan connection of a path structure}\label{section:Cartan}

Path geometries are treated in detail in section 8.6 of \cite{IL} and in \cite{BGH}.
Le $M$ be a real three dimensional manifold and $TM$ be its tangent bundle.

\begin{dfn}A path structure on $M$ is a choice of two subbundles $T^1$ and $T^2$ in $TM$ such that 
$T^1 \cap T^2=\{ 0\}$ and such that $T^1 \oplus T^2$ is a contact distribution.
\end{dfn}

The condition that $T^1 \oplus T^2$ be a contact distribution means that, locally,  there exists a one form $\theta\in {T^*M}$ such that
$\ker \theta= T^1\oplus T^2$ and $d\theta \wedge \theta$ is never zero.

Flat path geometry   is the geometry of real flags in $\R^3$.  That is the geometry of the space of all couples $(p,l)$ where $p\in \R P^2$ and $l$ is a real projective line
containing $p$.  The space of flags is identified to the quotient
$$
\SL(3,\R)/B_\R
$$
where $B_\R$ is the Borel group of all real upper triangular matrices.  
 The Maurer-Cartan form on $\SL(3,\R)$ is given by a form  with values  in the Lie algebra ${\mathfrak {sl}}(3,\R)$ :

$$\pi=\left ( \begin{array}{ccc}
                        \phi+w     &   \phi^2   &   \psi  \\

                     \omega^1          &   -2w   &    \phi^1 \\

                      \omega         &  \omega^{2}   &      -\phi+w
                \end{array} \right )
$$
satisfying the equation $d\pi+\pi\wedge \pi =0$.  That is 
$$
d\omega = \omega^1\wedge \omega^2 +2\phi\wedge \omega
$$
$$
d\omega^1=\phi \wedge  \omega^1+3w\wedge \omega^1 +\omega \wedge \phi^1
$$
$$
d\omega^2=\phi\wedge \omega^2-3w\wedge \omega^2 -\omega\wedge \phi^2
$$
$$
dw= -\frac{1}{2}\phi^2\wedge \omega^1+\frac{1}{2}\phi^1\wedge \omega^2
$$
$$
d\phi = \omega \wedge \psi -\frac{1}{2}\phi^2\wedge \omega^1-\frac{1}{2}\phi^1\wedge \omega^2
$$
$$
d\phi^1 = \psi\wedge \omega^1-\phi\wedge \phi^1+3w\wedge \phi^1
$$
$$
d\phi^2 = -\psi\wedge \omega^2-\phi\wedge \phi^2-3w\wedge \phi^2
$$
$$
d\psi= \phi^1\wedge \phi^2+2\psi\wedge \phi.
$$

\subsection{The coframe bundle $Y$ over the bundle $E$ of contact forms}\label{section:coframes}

We recall the construction of  the $\R^*$-bundle of contact forms (we follow here \cite{FVflag} and \cite{FVODE}). 
 
Define ${E}$ to be the $\R^*$-bundle of all forms $\theta$ on $TM$ such that $\ker \theta=T^1\oplus T^2$.   Remark that this bundle is trivial if and only if there   exists a globally defined 
non-vanishing form $\theta$. Define the set of forms $\theta^1$ and $\theta^2$ on $M$ satisfying
$$
\theta^1(T^1)\neq 0 \  
{\mbox{and}}\ 
\theta^2(T^2)\neq 0,
$$
$$
\ker \theta^1_{\vert\ker\theta}=T^2  ~\mbox{and}
 \  \ker \theta^2_{\vert\ker\theta}=T^1.
$$
Henceforth we fix one such choice, and all others are given by $\theta'^{i} = a^i \theta^i + v^i\theta$.

 On $E$ we define the tautological
form $\omega$.  That is $\omega_\theta=\pi^*(\theta)$ where
$\pi: {E}\rightarrow M$ is the natural projection.  
We also consider the tautological forms defined
by the forms $\theta^1$ and $\theta^2$ over the line bundle $E$.  That is, for each 
$\theta\in E$  we let $\omega^i_\theta= \pi^*(\theta^i)$.
At each point $\theta\in E$ we have the family of forms defined on $E$:
$$
        \omega' = \omega
$$
$$
        \omega'^{1} = a^1 \omega^1 + v^1\omega
$$
$$
        \omega'^{2} = a^2 \omega^2 + v^2\omega
$$

 We may, moreover, suppose that 
 $$
 d\theta= \theta^1\wedge \theta^2 \ \ \mbox{modulo} \ \theta
 $$ 
 and therefore
  $$
 d\omega= \omega^1\wedge \omega^2 \ \ \mbox{modulo} \ \omega .
 $$ 
  This imposes that $a^1a^2=1$.

  Those forms
vanish on vertical vectors, that is, vectors in the kernel
of the map $TE\rightarrow TM$. In order to define non-horizontal 1-forms 
we let $\theta$ be a section of $E$ over $M$ and introduce the coordinate
$\lambda\in \R^*$ in $E$.
By abuse of notation,
let $\theta$ denote the tautological form on the section $\theta$.
We write then the tautological form $\omega$ over $E$ as 
$$
\omega_{\lambda\theta}=\lambda \theta.
$$
Differentiating this formula we obtain
\begin{equation}  \label{domega}
                d\omega = 2\phi\wedge \omega + 
\omega^1\wedge\omega^{2}  
\end{equation}
where $\phi= \frac{d\lambda}{2\lambda}$ 
 modulo $\omega, \omega^1, 
\omega^{2}$.
Here
$
\frac{d\lambda}{2\lambda}
$ is a form intrinsically defined on $E$ up to horizontal
forms.
We obtain in this way a coframe bundle  satisfying equation \ref{domega} over $E$.  The coframes at each point of $E$ are given by
$$
        \omega' = \omega
$$
$$
        \omega'^{1} = a^1\omega^{1} + v^1 \omega
$$
$$
        \omega'^{2} = a^2\omega^{2} + v^2 \omega
$$
$$
        \phi'= \phi -\frac{1}{2}a^1v^{2}\omega^1 +\frac{1}{2}a^2v^1\omega^{2} +s\omega
$$
$v^1, v^2,s \in \R$  and $a^1,a^2\in \R^*$ such that $a^1a^2=1$.

\begin{dfn}
We denote by $Y$ the  coframe bundle
 $Y\rightarrow E$ given by the set of 1-forms  $\omega, \omega^{1},
  \omega^{2}, \phi$ as above.
  Two coframes
are related by
$$
(\omega', \omega'^{1},
  \omega'^{2}, \phi')=
(\omega, \omega^{1},
  \omega^{2}, \phi)
 \left ( \begin{array}{cccc}

                        1       &    v^1                &       v^2 & s \\

                        0 &     a^1 & 0            &       -\frac{1}{2}a^1 v^2  \\
                        0   &  0 & a^2             &       \frac{1}{2}a^2v^{ 1} \\
                        0      &  0 & 0     &       1

                \end{array} \right )
$$
where  and $s, v^1, v^2 \in \R$ and $a^1,a^2\in \R^*$ satisfy $a^1a^2=1$.
\end{dfn}

The bundle $Y$ can also be fibered over the manifold $M$.  In order to describe the bundle $Y$ as a principal fiber bundle over
 $M$ observe that choosing a local section $\theta$ of $E$ and forms $\theta^1$ and $\theta^2$ on $M$ such that
 $d\theta=\theta^1\wedge \theta^2$ one can write a trivialization of the fiber bundle as

$$
      \omega=  \lambda \theta
$$
$$
      \omega^1=   a^1\theta^{1} + v^1\lambda\theta
$$
$$
       \omega^2= a^2\theta^{2} + v^2\lambda\theta
$$
$$
       \phi= \frac{d\lambda}{2\lambda} -\frac{1}{2}a^1v^{2}\theta^1 +\frac{1}{2}a^2v^1\theta^{2} +s\theta,
         $$
where $v^1, v^2,s \in \R$  and $a^1,a^2\in \R^*$ such that $a^1a^2=\lambda$.  Here the coframe $\omega, \omega^1,\omega^2,\phi$ is seen as 
composed of tautological forms.
The group $H$ acting on the right of this bundle 
 is 
$$
H=\left\{ \left ( \begin{array}{cccc}

                        \lambda   	&    	v^1 \lambda      	&       v^2\lambda 	& 		s \\

                       0				 &     a^1 			&		 0            &      - \frac{1}{2}a^1 v^2\\
                       0   				&  	0 				&		 a^2        &        \frac{1}{2}a^2v^{ 1}\\
                        0		       &  		0              &     0   			&       1

                \end{array} \right )
\mbox{
where   $s, v^1, v^2 \in \R$ and $a^1,a^2\in \R^*$ satisfy $a^1a^2=\lambda$
}
\right\}.
$$

The bundle $Y\rightarrow M$ is a principal bundle with structure group $H$ which can be identified 
to the Borel group 
$B\subset \SL(3,\R)$ of upper triangular matrices  with determinant one via the map
$$
j: B\rightarrow H
$$
given by

$$
\left( \begin{array}{ccc}

                        a  	&    	c      	&       e \\

                       0				 &     \frac{1}{ab}		&		f\\
                       0   				&  	0 				&		b
                      
       \end{array} \right ) 
       \longrightarrow  
       \left ( \begin{array}{cccc}

                        \frac{a}{b}  		&    	-a^2 f    	&       	\frac{c}{b}		& 		-eb+\frac 1 2acf \\

                       0				&     a^2b 		&		0            		&      		-\frac 1 2 abc\\
                       0   			&  	0 		&		\frac{1}{ab^2}    &        	-\frac{f}{2b}\\
                       0		         	&  	0              	&       	0   			&       	1

                \end{array} \right )
$$

\subsection{The connection form on the bundle $Y$}\label{section:connectionforms}

Here we review the $   {\mathfrak {sl}}(3,\R)$-valued Cartan connection defined on the  coframe bundle $Y\rightarrow E$  as described in \cite{FVflag,FVODE}.
  
  Each point in the  coframe bundle $Y$ over $E$ is lifted to a  family of tautological forms on $T^*Y$.  This family is then completed to obtain a coframe bundle over $Y$ by an appropriate choice of conditions.  As usual, the conditions are essentially curvature conditions and  
  are obtained by differentiating the tautological forms and introducing new linearly independent forms satisfying certain canonical equations.  We state the final existence theorem of the adapted Cartan connection:

 \begin{thm} There exists a unique ${\mathfrak {sl}}(3,\R)$ valued connection form on the bundle $Y$
 $$\pi=\left ( \begin{array}{ccc}
                        \phi+w     &   \phi^{ 2}   &   \psi  \\

                     \omega^1          &   -2 w    &    \phi^1  \\

                      \omega         &  \omega^{2}   &      -\phi+w
                \end{array} \right ),
$$
whose curvature satisfies
$$\Pi= d\pi+\pi\wedge\pi=\left ( \begin{array}{ccc}
                        0 &   \Phi^{2}      &  \Psi   \\

                         0  &  0    &  \Phi^1  \\

                           0  &   0  &    0
               \end{array} \right )
$$
with $
\Phi^1=Q^1\omega\wedge \omega^2,\ \ 
\Phi^2=Q^2\omega\wedge \omega^1 \ 
\mbox{and}\ \ 
\Psi =\left( U_1\omega^1+ U_2\omega^2\right) \wedge \omega,
$
for functions $Q^1, Q^2, U^1$ and $U^2$  on $Y$.

 \end{thm}
 
 Writing the components of the curvature form explicitly we obtain the following equations:

 \begin{equation}  
                d\omega = 2\phi\wedge \omega + 
\omega^1\wedge\omega^{2}  
\end{equation}
 \begin{equation}\label{domegai}
d\omega^1=\phi \wedge  \omega^1+
3w \wedge \omega^1 +\omega\wedge \phi^1\  \mbox{and}\ \ 
d\omega^2=\phi \wedge  \omega^2 -
3 w\wedge \omega^2-\omega\wedge \phi^2
\end{equation}
\begin{equation}\label{dphi}
d\phi = \omega \wedge \psi-\frac{1}{2}(
\phi^{2}\wedge \omega ^1+
\phi^1\wedge \omega^{2})
\end{equation}
 \begin{equation}\label{omega11}
dw+\frac{1}{2} \omega^2\wedge \phi^1 -\frac{1}{2} \omega^1\wedge \phi^2=0,
\end{equation}
 \begin{equation}\label{Phi1}
\Phi^1= d\phi^1 +3\phi^1\wedge w+\omega^1\wedge \psi+\phi\wedge \phi^1  =Q^1\omega\wedge \omega^2,
\end{equation}
 \begin{equation}\label{Phi2}
\Phi^2=d\phi^2 -3\phi^2\wedge w-\omega^2\wedge \psi+\phi\wedge \phi^2=Q^2\omega\wedge \omega^1,
\end{equation}
 \begin{equation}\label{dPsi'}
\Psi := d\psi-\phi^1\wedge \phi^2+2\phi\wedge \psi =(U_1\omega^1+ U_2\omega^2)\wedge \omega.
 \end{equation}
where  $Q^1, Q^2, U^1$ and $U^2$ are functions on $Y$.

 \subsection{Bianchi identities}
 
 In this section we compute Bianchi identities.  They are essential to obtain relations between the curvatures and its derivatives
 and will be heavily used in the reductions of the path structures.

  \subsubsection{}Equation $d(d\phi^1)=0$ obtained differentiating $\Phi^1$ (equation \ref{Phi1}) implies
    \begin{equation}\label{dq1}
   dQ^1 -6Q^1 w +4Q^1\phi=S^1\omega +U_2\omega^1 + T^1\omega^2,
   \end{equation}
where we introduced functions $S^1$ and $T^1$.

   \subsubsection{}
   Anagously, equation $d(d\phi^2)=0$ obtained differentiating $\Phi^2$ (equation \ref{Phi2})
   implies
   \begin{equation}\label{dq2}
   dQ^2 +6Q^2 w +4Q^2\phi=S^2\omega -U_1\omega^2 + T^2\omega^1,
   \end{equation}
   where we introduced new functions $S^2$ and $T^2$.

\subsubsection{} Equation  $d(d\psi)=0$ obtained from equation 
\ref{dPsi'}
implies
 \begin{equation}\label{du1}
 dU_1+5U_1\phi +3U_1 w +Q^2\phi^1=A\omega+B\omega^1+C\omega^2
 \end{equation}
 and
  \begin{equation}\label{du2}
 dU_2+5U_2\phi -3U_2 w -Q^1\phi^2=D\omega+C\omega^1+E\omega^2.
 \end{equation}

\subsubsection{Higher order Bianchi identities}
If we derive equation \ref{dq1}
  and replace the known terms we get  
    \begin{equation}\label{ds1}
   dS^1+6S^1\phi-6S^1 w-U_2\phi^1+T^1\phi^2+4Q^1\psi=X_0\omega+ D\omega^1 +X_2\omega^2
   \end{equation}
    \begin{equation}\label{dt1}
   dT^1+5 T^1\phi-9T^1 w+5Q^1\phi^1=X_2\omega-( S^1- E) \omega^1+Y_2\omega^2
   \end{equation}
   In the same way, if we derive equation \ref{dq2}
  and replace the known terms we get  
 \begin{equation}\label{ds2}
  dS^2+6S^2\phi+6S^2 w-T^2\phi^1-U_1\phi^2+4Q^2\psi=Y_0\omega+Y_1\omega^1 - A\omega^2 
   \end{equation}
    \begin{equation}\label{dt2}
   dT^2+5  T^2\phi+9T^2 w+5Q^2\phi^2=Y_1\omega+X_1\omega^1+( S^2-B) \omega^2
   \end{equation}
If we differentiate equation \ref{dt2}, and use equations \ref{ds2}, \ref{dt2} and \ref{dq2} we get
\begin{equation}\label{ddt2}
\begin{array}{lr}
\omega\wedge(dY_1+7Y_1\phi+9Y_1w-X_1\phi^1+(6S^2-B)\phi^2 +5T^2\psi+5(Q^2)^2\omega^1-Y_0\omega^2)+&\\
\omega^1\wedge(dX_1+6X_1\phi+12X_1w+12T^2\phi^2)-&\\
\omega^2\wedge(dB+6B\phi+6Bw+T^2\phi^1+4U_1\phi^2 -Q^2\psi-2Y_1\omega^1)=0,&
\end{array}
\end{equation}
and from this we get
\begin{equation}\label{dX1}
dX_1+6X_1\phi+12X_1w+12T^2\phi^2=X_{10}\omega+X_{11}\omega^1+X_{12}\omega^2
\end{equation}
\begin{equation}\label{dY1}
dY_1+7Y_1\phi+9Y_1w-X_1\phi^1+(6S^2-B)\phi^2 +5T^2\psi+5(Q^2)^2\omega^1-Y_0\omega^2=Y_{10}\omega+X_{10}\omega^1+Y_{12}\omega^2
\end{equation}

\section{Strict path structures}\label{section:pseudo}
 
 In this section we recall the definition of strict path structures (see \cite{FMMV} and \cite{FVflag}).  They correspond to path structures with a fixed contact form.

 $G_1$ denotes from now on the subgroup of $\SL(3,\R)$ defined by
 $$
 G_1=
 \left\{\left ( \begin{array}{ccc}

                        a     &    0    &   0   \\
                       x    &    \frac{1}{a^2}    &   0\\

                        z       &  y &    a

                \end{array} \right )\ \vert\ {a\in\R^*,(x,y,z)\in\R^3}\right\}
$$ 
and $P_1\subset G_1$ the subgroup defined by
$$
 P_1=\left\{\left ( \begin{array}{ccc}

                        a     &    0    &   0   \\
                       0    &    \frac{1}{a^2}    &   0\\

                        0      &     0 &     a

                \end{array} \right )\right\}.\ \ \
$$
 Writing the Maurer-Cartan form of $G_1$ as the matrix
$$
\left ( \begin{array}{ccc}

                        w     &    0    &   0   \\
                       \theta^1    &    -2w   &   0\\

                       \theta      &  \theta^2   &     w

                \end{array} \right )
$$
one obtains the Maurer-Cartan equations:
$$d \theta+\theta^2\wedge \theta^1=0$$
$$d\theta^1 -3w\wedge \theta^1=0$$
$$d\theta^2 +3w\wedge \theta^2=0$$
$$d w=0.$$

Let $M$ be a three-manifold equipped with a path structure $D=E^1\oplus E^2$.  Fixing a contact form  $\boldsymbol\theta$ such that
$\ker \boldsymbol\theta=D$ defines a strict path structure.
Let $\boldsymbol R$ be the Reeb vector field associated to $\boldsymbol\theta$.  That is $\boldsymbol\theta(\boldsymbol R)=1$ and $d\boldsymbol\theta(\boldsymbol R,\cdot)=0$.
Let  $\boldsymbol X_1\in E^1$, $\boldsymbol X_2\in E^2$ be such that $d\theta(\boldsymbol X_1,\boldsymbol X_2)=1$.  
The dual coframe of $(\boldsymbol X_1,\boldsymbol X_2,\boldsymbol R)$ is $(\boldsymbol\theta^1,\boldsymbol\theta^2,\boldsymbol\theta)$,
for two 1-forms $\boldsymbol\theta_1$ and $\boldsymbol\theta_2$ verifying $d\boldsymbol\theta=\boldsymbol\theta^1\wedge \boldsymbol\theta^2$.   

At any point $x\in M$,
one can look at the coframes of the form
$$
{\theta}^1={a^3}\boldsymbol\theta^1(x), \  \theta^2=\frac{1}{a^3}\boldsymbol\theta^2(x), \ \theta=\boldsymbol\theta(x)
$$
for $a\in\R^*$.
\begin{dfn}
We denote by $\pi: Y^1\rightarrow M$ the  {$\R^*$}-coframe bundle over $M$ 
  given by the set of coframes  $(\theta, \theta^{1}, \theta^{2})$ of the above form.
\end{dfn}

We will denote the tautological forms defined by $\theta^1,\theta^2,\theta$ using the same letters. 
That is, we write $\theta^i$ at the coframe $(\theta^1,\theta^2,\theta)$ to be $\pi^*(\theta^i)$.

\begin{prop}\label{cartan-strict}
There exists a unique $\mathfrak{g}_1$-valued Cartan connection on $Y^1$
$$
\varpi=\left ( \begin{array}{ccc}

                        \upsilon     &    0    &   0   \\
                       \theta^1    &    -2 \upsilon   &   0\\

                       \theta      &  \theta^2   &      \upsilon

                \end{array} \right )
$$ such that its curvature form is of the form
$$
\varpi=\left ( \begin{array}{ccc}

                        d\upsilon     &    0    &   0   \\
                       \theta\wedge \tau^1    &    -2d\upsilon  &   0\\

                       0      &  \theta\wedge \tau^2   &    d\upsilon

                \end{array} \right )
$$
with $\tau^1\wedge \theta^2=\tau^2\wedge \theta^1=0$.
\end{prop}

Observe that the condition $\tau^1\wedge \theta^2=\tau^2\wedge \theta^1=0$ implies that we may write
$\tau^1=\tau^1_2\theta^2$ and $\tau^2=\tau^2_1\theta^1$.  The structure equations are
\begin{equation}\label{equations:strict}
\begin{array}{lcl}
d\theta^1 -3\upsilon\wedge \theta^1 & =&\theta\wedge \tau^1,\\
d\theta^2 +3\upsilon\wedge \theta^2&=&\theta\wedge \tau^2,\\
d\theta&=&\theta^1\wedge \theta^2.
\end{array}
\end{equation}

A choice of coframe $(\boldsymbol\theta^1,\boldsymbol\theta^2,\boldsymbol\theta)$  on a strict path structure  defines an embedding
$$
j : M\to Y^1
$$
and therefore 
\begin{equation}\label{equations:jembedding}
\begin{array}{lcl}
d\boldsymbol\theta^1 -3j^*\upsilon\wedge \boldsymbol\theta^1 & =&\boldsymbol\theta \wedge j^*\tau^1,\\
d\boldsymbol\theta^2 +3j^*\upsilon\wedge \boldsymbol\theta^2&=&\boldsymbol\theta\wedge j^*\tau^2,\\
d\boldsymbol\theta&=&\boldsymbol\theta^1\wedge \boldsymbol\theta^2.
\end{array}
\end{equation}

\subsection{Bianchi identities}
In what follows,
the equations should be understood as definitions for
the coefficients appearing in the right hand terms.
Bianchi identities give the following equations:

 \begin{equation}\label{dtheta11}
d \upsilon=R\theta^1\wedge \theta^2+W^1\theta^1\wedge\theta+W^2\theta^2\wedge\theta
\end{equation}
\begin{equation}\label{dtau1}
d\tau^1+3\tau^1\wedge \upsilon=3W^2\theta^1\wedge \theta^2+S^1_1\theta\wedge\theta^1+S^1_2\theta\wedge\theta^2
\end{equation}
\begin{equation}\label{dtau2}
d\tau^2-3\tau^2\wedge \upsilon=3W^1\theta^1\wedge \theta^2+S^2_1\theta\wedge\theta^1+S^2_2\theta\wedge\theta^2
\end{equation}
Moreover, we have the relation
$$
S^1_1=S^2_2=\tau^1_2\tau^2_1.
$$

 From equation $dd\upsilon=0$ one obtains that 
 $$
 dR=R_0\theta +R_1\theta^1+R_2\theta^2,
 $$
\begin{equation}\label{dW1}
 dW^1+3W^1\upsilon=W^1_0\theta+W^1_1\theta^1+W^1_2\theta^2
 \end{equation}
 and
 \begin{equation}\label{dW2}
  dW^2-3W^2\upsilon=W^2_0\theta+W^2_1\theta^1+W^2_2\theta^2
 \end{equation}

 with
 $$
 R_0=W^1_2-W^2_1.
 $$

 Also, 
 writing $dR_0=R_{00}\theta+R_{01}\theta^1+R_{02}\theta^2$, one gets
\begin{equation}\label{dR1}
dR_1+3R_1\upsilon+R_2\tau^2_1\theta-\frac{1}{2}R_0\theta^2=R_{01}\theta+R_{11}\theta^1+R_{12}\theta^2
 \end{equation}
 and
 \begin{equation}\label{dR2}
dR_2-3R_2\upsilon+R_1\tau^1_2\theta+\frac{1}{2}R_0\theta^1=R_{02}\theta+R_{12}\theta^1+R_{22}\theta^2.
 \end{equation}
 We have moreover
 $$
 d\tau^1_2-6\tau^1_2\upsilon=3W^2\theta^1 +S^1_2\theta \mod \theta^2
 $$
 and
 $$
 d\tau^2_1+6\tau^2_1\upsilon=-3W^1\theta^2 +S^2_1\theta \mod \theta^1.
 $$



\section{Reductions}\label{section:reduction}

In this section we prove the main theorem concerning canonical reductions of a path structure.  The motivation behind the reductions  is a gap theorem on the possible dimensions of the automorphism group.  Indeed,
the group of transformations preserving a path structure which is not flat has dimension at most three:   

\begin{thm}[M. A. Tresse \cite{tresse}]
The group of transformations of the fiber bundle $Y$ has  dimension eight (in the flat case) or at most dimension three.
\end{thm}

A modern proof is contained in M.  Mion-Mouton thesis \cite{MM} and a  more general result is obtained in  \cite{gapphenomenon}.

We will describe in this section reductions of the fiber bundle $Y$ to a parallelism or a $\Z/2\Z$-bundle over the manifold in the case the path structure is not flat.  This clearly implies Tresse's theorem. 

\begin{thm}\label{thm-reduction}
If the path structure is not flat, there exists a canonical reduction of the fiber bundle $Y$ to  a $\Z/2\Z$-structure.  Moreover,
\begin{enumerate}
\item if  $Q^1\neq 0$, and $Q^2\neq 0$, and $T_1\neq 0$ or $T_2\neq 0$, there exists a further reduction to a parallelism.
\item  if  $Q^1= 0$ and $Q^2\neq 0$ and $Y_1\neq 0$  there exists a further reduction to a parallelism.
\end{enumerate}
Here $Q^1$, $Q^2$,
$T_1$, $T_2$ and  $Y_1$ are the functions on the bundle $Y$ introduced in equations
\ref{Phi1}, \ref{Phi2},
\ref{dq1}, \ref{dq2} and \ref{ds2}.
\end{thm}

The theorem is a consequence of propositions  \ref{rq1q2}, \ref{rq2}, \ref{proposition:reduction}, \ref{pg} and \ref{rY}, where details of the reductions are given.  Note that the case  $Q^1\neq 0$ and $Q^2= 0$ differs from the second case in the theorem by an ordering of the 
decomposition of the plane field and the result is analogous.

\subsection{Reduction of the structure: $Q^1\neq 0$ and $Q^2\neq 0$.}

Suppose there exists a section of the coframe bundle $Y$ with $Q^1\neq 0$ and $Q^2\neq 0$. From  equations \ref{transformationsQ}, $ab^5 Q_1=\tilde Q_1$, $a^5b\tilde Q_2=Q_2$,  we can solve for $a,b$ such  that $\tilde Q_1=1$ and $\tilde Q_2=\epsilon$, where $Q_1Q_2=\epsilon |Q_1Q_2|$.   Observe that if we only consider coframes on $Y$ satisfying $ Q_1=1$, $ Q_2=\epsilon$, then $a=b=\pm 1$.

From transformation properties \ref{transformationsU}:
\begin{equation}\label{U1}
\tilde U_1=\frac{b}{a^4}(U_1-\frac{f}{b}Q^2)
\end{equation}
and
$$
\tilde U_2=\frac{b^4}{a}( U_2+abcQ^1),
$$
we can choose $c=-\frac{U_2}{abQ^1}$ and $f=\frac{bU_1}{Q^2}$ such that $\tilde U_1=\tilde U_2=0$.
 Then, from equations \ref{dq1} and \ref{dq2}, that is
$$
 d\tilde Q^1 -6\tilde Q^1\tilde w +4\tilde Q^1\tilde\phi=\tilde S^1\tilde\omega +\tilde U_2\tilde \omega^1 + \tilde T^1\tilde \omega^2
$$
and
$$
   d\tilde Q^2 +6\tilde Q^2 \tilde w +4\tilde \tilde Q^2\tilde \phi=\tilde S^2\tilde \omega -\tilde U_1\tilde \omega^2 + \tilde T^2\tilde \omega^1,
$$
we get
\begin{equation}\label{f-w}
4\tilde \phi -6\tilde w=\tilde S^1\tilde\omega  + \tilde T^1\tilde \omega^2,
\end{equation}
\begin{equation}\label{f+w}
 \epsilon(4\tilde \phi+6\tilde w )=\tilde S^2\tilde \omega + \tilde T^2\tilde \omega^1.
\end{equation}

Consider now only the coframes on $Y$ satisfying $ Q_1=1$, $ Q_2=\epsilon$ and $ U_1= U_2=0$.
Using formulas \ref{adjunta} with $a=b=\pm 1, c=f=0$ (these parameters are fixed by the previous conditions) we obtain that
$\tilde w=w$ and $\tilde \phi=\phi-e\omega$.
Therefore

$$
4( \phi -e\omega) -6w=\tilde S^1\omega  + \tilde T^1\omega^2
$$
and
$$
4( \phi -e\omega) +6w=\tilde S^2 \omega + \tilde T^2 \omega^1.
$$
From these equations we observe that $\phi$ and $w$ are combinations of 
$\omega, \omega^1$ and $ \omega^2$.
We also obtain  that 
$$
\tilde S^1=S^1 +4e
$$
and
$$
\tilde S^2=S^2 +\epsilon4e.
$$

\begin{prop}\label{rq1q2}
Suppose there is a section of $Y$ such that $Q^1\neq 0$ and  $Q^2\neq 0$.  Then there exists a unique $\Z/2\Z$-bundle $Y^{red}\subset Y$ such that on $Y^{red}$, 
$Q^1=1$, $Q^2=\epsilon$, 
$U_1=U_2=0$ and $S^1+\epsilon S^2=0$, where  $Q^1Q^2 = \epsilon |Q^1Q^2|$.  Moreover, if $T_1\neq 0$ or $T_2\neq 0$ there is a further reduction to a parallelism.
\end{prop}

\Pf
Observe that $\tilde S^1+\epsilon \tilde S^2=S^1+\epsilon S^2+8e$ and choose the unique $e$
satisfying $\tilde S^1+\epsilon \tilde S^2=0$.  This defines the reduction to a $\Z/2\Z$-bundle $Y^{red}\subset Y$.

From \ref{f-w} and \ref{f+w} we get on $Y^{red}$
$$
\phi=\frac 1 8(T^1\omega^2+\epsilon T^2\omega^1)
$$
and
$$
w=\frac{1}{12}((\epsilon S^2-S^1) \omega+\epsilon T^2\omega^1-T^1\omega^2)
$$
Since both forms are invariant on $Y^{red}$, fixing the sign of $T^1$ or $T^2$ determines one of the two
coframes. That is, either $(\omega^1, \omega^2,\omega)$ or  $(-\omega^1, -\omega^2,\omega)$.

\EPf

The structure equations of a parallelism can be written as follows.
From equations \ref{du1} and \ref{du2} we get on $Y^{red}$
$$
 \phi^1=\epsilon(A\omega+B\omega^1+C\omega^2)
 $$
  and
$$
   \phi^2=-(D\omega+C\omega^1+E\omega^2).
 $$
Therefore the structure equations are
 \begin{equation}\label{dif1}
 d\omega=\frac 1 4( T^1\omega^2+\epsilon T^2\omega^1)\wedge\omega+\omega^1\wedge\omega^2, 
 \end{equation}
 \begin{equation}\label{dif2}
 d\omega^1=\frac 1 8 T^1\omega^1\wedge\omega^2+(\frac 1 4(\epsilon S^2-S^1)+\epsilon B)\omega\wedge\omega^1+\epsilon C \omega\wedge\omega^2
 \end{equation}
 \begin{equation}\label{dif3}
 d\omega^2=-\epsilon\frac 1 8 T^2\omega^1\wedge\omega^2+C \omega\wedge\omega^1+(-\frac 1 4(\epsilon S^2-S^1)+E)\omega\wedge\omega^2.
 \end{equation}

As a final observation, other possible reductions to a $\mathbb Z^2$-bundle could be made. For instance, imposing the condition $S^1=0$ defines such a reduction. In any case, the theorem shows that 
in the case both $Q^1$ and $Q^2$ are non-vanishing, the dimension of the automorphism group 
is at most three. 

\subsection{Reduction of the structure: $Q^1=0$ and $Q^2\neq 0$.}

Suppose there exists a section of the coframe bundle $Y$ with $Q^1=0$ and $Q^2\neq 0$.
From equation \ref{dq1},
$$
  dQ^1 -6Q^1 w +4Q^1\phi=S^1\omega +U_2\omega^1 + T^1\omega^2,
$$   
 we obtain that $S^1=U_2=T^1=0$ on $Y$. Also, from equation \ref{du2} we obtain $D=C=E=0$,  and from \ref{ds1} and \ref{dt1}, 
 $X_0=X_2=Y_2=0$.

 Using the transformation $\tilde Q^2=\frac{1}{a^5b}Q^2$ one can chose an arbitrary function $a$ and then the function $b$
 such that  $\frac{1}{a^5b}Q^2=1$ is determined.  
 One considers now the subbundle with  coframes satisfying $Q^2=1$. Its structure group satisfies then $a^5=\frac{1}{b}$. From equation \ref{U1}, as in the previous section,
\begin{equation}
\tilde U_1=\frac{b}{a^4}(U_1-\frac{f}{b}Q^2)=\frac{b}{a^4}(U_1-fa^5),
\end{equation}
 and one can choose $f$ (depending on $a$) such that $U_1=0$.
 
 The structure group of this reduction consists of matrices of the form
 $$
\left( \begin{array}{ccc}

                        a  	&    	c      	&       e \\

                       0				 &     {a^4}		&		0\\
                       0   				&  	0 				&		\frac{1}{a^5}
                      
       \end{array} \right )
$$
Compute now the transformation by a section of this structure group of the pull-back of the connection forms  to this subbundle.
We have $ R^*_h(\pi)=h^{-1}dh+h^{-1}\pi h$.  The computation of $h^{-1}\pi h$ is given above in formulae  \ref{adjunta}.   We also compute
\begin{equation}\label{hdh}
h^{-1}dh=\left( \begin{array}{ccc}

                        a^{-1}da  	&    	\star      	&       \star \\

                       0				 &    -a^{-1}da -b^{-1}db		&		\star\\
                       0   				&  	0 				&		b^{-1}db
                      
       \end{array} \right ).
\end{equation}
Therefore, taking into account that $b=\frac{1}{a^5}$ (so $b^{-1}db=-5\frac{da}{a}$) and $f=0$ in \ref{adjunta},
$$
\tilde w =R^*_h(w)=\frac{1}{2}(a^{-1}da +b^{-1}db)+ w-\frac{1}{2}abc\,\omega^1=-2a^{-1}da + w-\frac{1}{2}ca^{-4}\,\omega^1
$$
and
$$
\tilde\phi= R^*_h(\phi)=\frac{1}{2}(a^{-1}da -b^{-1}db) +\phi-\frac{1}{2}abc\,\omega^1-\frac{e}{b}\,\omega=3a^{-1}da +\phi-\frac{1}{2}ca^{-4}\,\omega^1-{e}a^5\,\omega.
$$

Observe now that

$$
6\tilde w+4\tilde\phi= 6w+4\phi -5ca^{-4}\,\omega^1-{e}a^5\,\omega
$$
Replacing equation \ref{dq2}:
$dQ^2 +6Q^2 w +4Q^2\phi=S^2\omega -U_1\omega^2 + T^2\omega^1$,
(imposing $Q^2=1$ and $U_1=0$) on the last equation we obtain
$$
6\tilde w+4\tilde\phi= (T^2 -5ca^{-4})\,\omega^1+(S^2-{e}a^5)\,\omega.
$$
We may now choose $c$ and $e$ so that 
$$
6\tilde w+4\tilde\phi= 0.
$$
That is, $\tilde T^2=\tilde S^2=0$.

The coframes with $Q^2=1$ and $6 w+4 \phi= 0$ form a subbundle $Y^1$ with group
$$
H_1=\left\{
\left( \begin{array}{ccc}

                        a  	&    	0      	&       0 \\

                       0				 &    a^4 		&		0\\
                       0   				&  	0 				&		a^{-5}
                      
       \end{array} \right ) \ \vert \ a\neq 0\ \right\}\subset H.
$$
From equation \ref{du1} we obtain that on $Y^1$ 
\begin{equation}\label{tildephi1}
\phi^1=A\omega+B\omega^1.
\end{equation}
From equation \ref{ds2} we obtain that on $Y^1$
\begin{equation}\label{tildepsi}
4\psi=Y_0\omega+Y_1\omega^1 - A\omega^2.
\end{equation}
Also, from equation \ref{dt2} we obtain on $Y^1$

\begin{equation}\label{tildephi2}
5\phi^2=Y_1\omega+X_1\omega^1-B \omega^2.
\end{equation}

We showed
\begin{prop}\label{rq2}
There exists a unique $\R^*$-bundle $Y^1\subset Y$ such that on $Y^1$, $Q^1=0$, 
$Q^2=1$, $U_1=0$ and $3 w=-2\phi$.  Moreover on $Y^1$ we have $C=D=E=S^1=S^2=T^1=T^2=U_2=0$ and $X_2=Y_2=0$, also $\phi^1=A\omega+B\omega^1$, $5\phi^2=Y_1\omega+X_1\omega^1-B \omega^2$ and $4\psi=Y_0\omega+Y_1\omega^1 - A\omega^2$.
\end{prop}

\subsubsection{Further reductions  of the structure: $Q^1=0$ and $Q^2\neq 0$.}

From equations \ref{adjunta}, taking into account Proposition \ref{rq2}, we obtain on $Y^1$ the transformations
\begin{align}
\tilde\omega  &= a^6\, \omega\notag\\
\tilde{\omega}^1 &=a^{-3}\, \omega^1\notag\\
\tilde\omega^2 &=a^{9}\,\omega^2\notag\\
\tilde\phi^1 &=b^2a\, \phi^1=a^{-9}\, \phi^1\notag\\
\tilde\phi^2 &=\frac{1}{ba^2}\, \phi^2=a^3\, \phi^2\notag\\
\tilde\psi &=\frac{b}{a}\,\psi=\frac{1}{a^6}\,\psi.
\end{align}

Therefore, from Proposition \ref{rq2}, we compute the transformations of $X_1,Y_0, Y_1,A$ and $B$:
For instance, from $\tilde \phi^2=a^3\, \phi^2$ we have
$$
\tilde Y_1a^{6} \omega+\tilde X_1 a^{-3} \omega^1-\tilde B a^{9} \omega^2=5\tilde \phi^2=a^3\, 5\phi^2=a^3(Y_1 \omega+X_1\omega^1-B\omega^2)
$$
and comparing the terms of  this equality we obtain

\begin{align}
\tilde X_1 &= a^6X_1\notag\\
\tilde Y_1 &= a^{-3}Y_1\notag\\
\tilde B &= a^{-6}B.
\end{align}

Analogously, we also obtain,  from $\tilde \psi=a^{-6}\, \psi$
$$
\tilde Y_0a^6\, \omega+\tilde Y_1a^{-3}\, \omega^1 -\tilde Aa^{9}\,\omega^2=4\tilde \psi=4\frac{1}{a^6}\,\psi=\frac{1}{a^6}(Y_0\omega+Y_1\omega^1 - A\omega^2)
$$
and therefore
\begin{align}
\tilde Y_0 &= a^{-{12}}Y_0\notag\\ 
\tilde A &= a^{-{15}}A.
\end{align}

Remark that if any of the coefficients $X_1,Y_0, Y_1,A$ or $B$ is non zero one could reduce the bundle by fixing the function $a$ (only up to a sign if $Y_1=0$ {{and $A=0$}}) so that the coefficient be constant equal to one.  Now we compute equation \ref{ddt2}, taking into account Proposition \ref{rq2},
and supposing that $X_1=Y_0= Y_1=A=0$, we obtain
$$
5\omega\wedge \omega^1=0,
$$
which is a contradiction.  We obtained therefore the following

\begin{prop} \label{proposition:reduction} If  $Q^1=0$ and $Q^2\neq 0$ there exists a canonical reduction of the path structure  to a $\Z/2\Z$-structure.  Moreover, if $Y_1\neq  0$  or $A\neq 0$, one can reduce further to a parallelism.
\end{prop}

A more refined description of the reduction is obtained by observing that the functions  $X_1$ and $ Y_1$ cannot both vanish.

\begin{prop}\label{pg}
On the fiber bundle $Y^1$ the functions $X_1$ and $ Y_1$ can not be simultaneously zero.
\end{prop}
\Pf On $Y^1$, taking account of theorem \ref{rq2}, we get from \ref{ddt2}
$$
\begin{array}{l}
\omega\wedge\left(dY_1+Y_1\phi+(5-X_1B)\omega^1+(\frac{3}{2}A-Y_0)\omega^2\right)+
\omega^1\wedge\left(dX_1-2X_1\phi-2Y_1\omega^2 \right)\\
-\omega^2\wedge\left(dB+2B\phi -\frac{1}{4}Y_0\omega-\frac{1}{4}Y_1\omega^1\right)=0.
\end{array}
$$
If $X_1=Y_1=0$, we obtain modulo $\omega^2$
$$
 5\omega\wedge \omega^1=0
$$
which is a contradiction.
\EPf

We can then use the function $a$ to fix $X_1=\pm1$ or $Y_1=1$.
Consider first the case $Y_1\neq 0$. We can use a section of $H^1$ to impose $\tilde Y_1=1$.
\begin{prop}\label{rY}
If $Y_1\neq 0$ we can choose $a=\left(Y_1\right)^{\frac 1 3}$,  and  with this choice we reduce $Y^1$ to  an $\{e\}$-bundle $Y^2$ where  $\tilde Y_1=1$, $\tilde Q^2=1$, $\tilde U_1=\tilde T^2=\tilde S^2=0$.  Moreover on $Y^2$ we have $\tilde Q^1=\tilde S^1=\tilde U_2=\tilde T^1=0$,
$3 \tilde w=-2\tilde \phi$, $\tilde \phi^1= \tilde A\tilde \omega+\tilde B\tilde \omega^1$, $5\tilde \phi^2=\tilde \omega+\tilde X_1\tilde \omega^1-\tilde B\tilde \omega^2$, $4\tilde \psi=\tilde Y_0\tilde \omega+\tilde \omega^1 - \tilde A\tilde \omega^2$, $\tilde \phi=(
\tilde A\tilde X_1+\frac{1}{5}\tilde B+\tilde Y_{10})\tilde \omega+(\frac{6}{5}\tilde X_1\tilde  B+X_{10}-5)\tilde \omega^1+(\tilde Y_0+\tilde Y_{12}-\frac{\tilde B^2}{5})\tilde \omega^2$.
\end{prop}
\Pf The choice of $a$ implies that $\tilde Y_1=1$. The other properties follow from theorem \ref{rq2} and from \ref{dY1}.\EPf

If $Y_1=0$, we can use $X_1\neq 0$ to reduce $Y^1$ to a $\Z/2\Z$-bundle. 

\begin{prop}\label{rX}	  
If $X_1\neq 0$, $Y_1=0$, we can choose 
$a= \pm|X_1|^{-\frac 1 6}$, and  with this choice we reduce $Y^1$ to  a $\Z/2\Z$-bundle $Y^2$ where $\tilde X_1=\epsilon$, with $\epsilon=\pm 1$, $\tilde Q^2=1$, $\tilde U_1=\tilde T^2=\tilde S^2=\tilde Y_1=0$.  Moreover on $Y^2$ we have $\tilde Q^1=\tilde S^1=\tilde U_2=\tilde T^1=0$,
$3 \tilde w=-2\tilde \phi$, $\tilde \phi^1= \tilde A\tilde \omega+\tilde B\tilde \omega^1$, $5\tilde \phi^2=\epsilon\tilde \omega^1-B\tilde \omega^2$, $4\tilde \psi=\tilde Y_0\tilde \omega - \tilde A\tilde \omega^2$, $-2\epsilon\tilde\phi=\tilde X_{10}\tilde \omega+\tilde X_{11}\tilde\omega^1+\tilde X_{12}\tilde \omega^2$
\end{prop}
\Pf The choice of $a$ implies  that $\tilde X_1=\epsilon$. The other properties follows from theorem \ref{rq2} and from \ref{dX1}.\EPf


\section{Homogeneous structures}\label{section:homogeneous}

\subsection{Three dimensional Lie groups with path structures}\label{subsection:3dim}

By Tresse's result, all homogeneous path structures which are not flat have a three dimensional group of automorphisms.  Therefore they are all 
locally isomorphic to a
left invariant structure on a three dimensional Lie group.
More precisely, we have the following direct Corollary
(Corollary \ref{corollaire-intro} of the introduction
that we state here again for the convenience of the reader).
A \emph{local Killing field} of a path structure
at a point $p$
is a vector field defined on a neighbourhood of $p$,
whose flow preserves the path structure.
\begin{cor}
At a point where the curvature is non-zero and
the algebra of local Killing fields
has dimension at least three,
a path structure is locally isomorphic to a
left-invariant path structure on a three-dimensional
Lie group,
and the algebra of local Killing fields
has dimension exactly three.
\end{cor}
\Pf
Let $p\in M^3$ be a point where the curvature of the path structure
$\mathcal{L}$
does not vanish, and where the Lie algebra
$\mathfrak{g}=\mathfrak{kill}^{loc}_{\mathcal{L}}(p)$
of (germs of) local Killing fields of $\mathcal{L}$ at $p$
has dimension at least three.
According to Theorem \ref{thm-reduction},
the Cartan bundle $Y$ admits a canonical reduction
to a $\mathbb{Z}/2\mathbb{Z}$-sub-bundle $Y'$.
In particular for any $V\in\mathfrak{g}$,
the canonical lift $\hat{\varphi}_V^t$ of
the flow of $V$ to $Y$
which preserves the Cartan connexion,
does preserve the reduction $Y'$.
Hence if $V(p)=0$,
then with $\hat{p}\in Y'$ any point of the fiber of $p$ in $Y'$,
the continuous map
$t\mapsto \hat{\varphi}_V^t(\hat{p})$ has values
in the discrete fiber of $p$ in $Y'$
and is therefore constant.
Therefore $\hat{\varphi}_V^t(\hat{p})=\hat{p}$
for any small enough times $t$, and
since $\hat{\varphi}_V^t$ preserves
the Cartan connexion which is a parallelism on $Y$,
this
forces $\hat{\varphi}_V^t$ to equal the identity on a
neighbourhood of $\hat{p}$ in $Y$.
The flow of $V$ is thus trivial on a neighbourhood of $p$,
\emph{i.e.} $V=0$ in $\mathfrak{g}$.
In the end, the evaluation map
$V\in\mathfrak{g}\mapsto V(p)$
is injective.
In particular, $\mathfrak{g}$ has dimension exactly three.
With $G$ the simply connected Lie group of Lie algebra
$\mathfrak{g}$
and $U_0$ an open neighbourhood of $0$ in $\mathfrak{g}$
on restriction to which the exponential map $\exp$
is a diffeomorphism onto its image $U'$,
this shows moreover that
the differential of
the map $Ev_p\colon\exp(V)\in U\mapsto\varphi_V^1(p)\in M^3$
at the identity is injective.
Since $\dim\mathfrak{g}\geq3=\dim M$,
$Ev_p$ is therefore by the inverse mapping theorem a
diffeomorphism onto an open neighbourhood $V$ of $p$,
possibly restricting $U_0$.
With $\mathcal{L}_{G}$
the unique path structure of $U$ such that
$Ev_p$ is a path structure isomorphism from
$(U,\mathcal{L}_{G})$
to $(V,\mathcal{L})$,
$\mathcal{L}_{G}$ is by construction invariant
by the left-invariant vector fields of $G$,
which allows to extend $\mathcal{L}_{G}$ to a left-invariant
path structure on $G$ (denoted in the same way).
This shows that the restriction of $\mathcal{L}$ to $V$
is isomorphic to the restriction of a left-invariant
path structure on a three-dimensional Lie group,
which concludes the proof.
\EPf

In this section we classify three dimensional Lie groups with homogeneous path structures. The results are summarized in Tables in section
\ref{subsection:table}.

We follow an analogous scheme to classify homogeneous CR structures (see \cite{FG}). In order to classify the simply connected groups having  left invariant path structure it is sufficient to classify three-dimensional Lie algebras $\mathfrak{g}$ admitting a vector subspace $\mathfrak{p}\subset \mathfrak{g}$ such that $\mathfrak{g}=[\mathfrak{p},\mathfrak{p}]\oplus\mathfrak{p}$ with a decomposition $\mathfrak{p}=\mathfrak{e_1}\oplus\mathfrak{e_2}$.

Given a basis $\{X_1,X_2\}$ of $\mathfrak{p}$, with $X_i\in \mathfrak{e_i}$, define $Y=-[X_1,X_2]$ and consider the map $\ad_Y:\mathfrak{p}\to \mathfrak{g}$ whose matrix in the given basis is
$$
A= \left ( \begin{array}{cc}

                        a_{11}      	&    a_{12}             	\\

                        a_{21}  	&      a_{22} 	      \\
                        
                       a_{31}  	&      a_{32} 	

                \end{array} \right ).
$$
Jacobi identities are computed as the following equations:
$$
 a_{11}      	+    a_{22}  =0
$$
$$
a_{31}a_{12}-a_{32}a_{11}=0
$$
$$
a_{31}a_{22}-a_{32}a_{21}=0.
$$
Note that in terms of a dual basis $\omega, \omega^1,\omega^2$ we obtain the equations
\begin{equation}\label{equations:homogeneous}
\begin{array}{lcl}
d\omega^1&=&a_{11}\omega^1\wedge\omega+a_{12}\omega^2\wedge\omega\\
d\omega^2&=& a_{21}\omega^1\wedge\omega+a_{22}\omega^2\wedge \omega\\
d\omega &=&a_{31}\omega^1\wedge\omega+a_{32}\omega^2\wedge \omega+\omega^1\wedge\omega^2
\end{array}
\end{equation}

A different basis $\{\bar{X_1},\bar{X_2}, \bar{Y}=-[\bar{X_1},\bar{X_2}]\}$ for $\mathfrak{g}$, which defines the same path structure, is given by the change of basis matrix
$$
P= \left ( \begin{array}{ccc}

                        \lambda_1     	 &    0   	&          0	\\

                       0				&      \lambda_2 	& 0      \\
                        
                       0 			&      0 		& \lambda_1\lambda_2

                \end{array} \right ),
$$with  $\lambda_1,\lambda_2\in \R^*$. 
The matrix of $\ad_{\bar{Y}}$ is now $\bar{A}=\lambda_1\lambda_2P^{-1}AN$, where 
$$
N=\left ( \begin{array}{cc}

                        \lambda_1     	 &    0   		\\

                       0				&      \lambda_2

                \end{array} \right )
$$
and therefore
 $$
 \bar{A}=\left ( \begin{array}{cc}

                        \lambda_1\lambda_2 a_{11}      	&     { \lambda_2}^2 a_{12}             	\\

                          {\lambda_1}^2a_{21}  	&       \lambda_1\lambda_2 a_{22} 	      \\
                        
                      \lambda_1 a_{31} 	&    \lambda_2  a_{32}	

                \end{array} \right ).
 $$

The goal now is to use the change of basis in order to find normal forms for $ad_Y$.  
Observe that the vector space isomorphism
$$
X_1\to X_2, \ \  X_2\to X_1,\ \  Y\to -Y
$$
changes the path structures. But we might not distinguish them as they correspond to a reordering of the decomposition. In other words,without loss of generality, we allow ourselves to change the order of the decomposition $\mathfrak{p}=\mathfrak{e_1}\oplus\mathfrak{e_2}$.
In this case, the matrix $A$ changes to
$$
 \bar{A}=\left ( \begin{array}{cc}

                       - a_{22}      	&      -a_{21}             	\\

                        - a_{12}  	&      -a_{11} 	      \\
                        
                       a_{32} 	&     a_{31}

                \end{array} \right ).
 $$

There are two cases to consider:

\subsubsection{$\ad_Y\mathfrak{p}\subset \mathfrak{p}$}
In this case  $a_{31}=  a_{32} =0$.	The general matrix for $\ad_Y$ is 
$$
\left ( \begin{array}{cc}

                        a     	&    b         	\\

                        c	&      -a	      \\
                        
                       0 	&      0	

                \end{array} \right ).
$$
\begin{itemize}
\item
If $a=b=c=0$ we obtain the usual invariant path structure on the Heisenberg group.

\item
If $b=c=0$ and $a\neq 0$ we may normalize the basis to obtain the normal form
$$
 \left ( \begin{array}{cc}

                        1     	&    0         	\\

                        0	&      -1	      \\
                        
                       0 	&      0	

                \end{array} \right ),
$$
which gives the usual  invariant path structure on $\SL(2,\R)$.

\item
If $b\neq 0$ (which we can assume also if  $c\neq 0$ by a change of order) we normalize the matrix of $\ad_Y$ to
$$
 \left ( \begin{array}{cc}

                        a     	&    \pm 1         	\\

                        c	&      -a	      \\
                        
                       0 	&      0	

                \end{array} \right )
$$
where $a,c\in \R$.  One can further normalize:
\begin{enumerate}
\item
If $a\neq 0$ then we normalize to
$$
 \left ( \begin{array}{cc}

                        1     	&    \pm 1         	\\

                        c	&      -1	      \\
                        
                       0 	&      0	

                \end{array} \right ).
$$
Here there are two cases.  If $c\neq \mp 1$ then the group is simple.  Indeed  if, $\pm c+1>0$ then it corresponds to $\SL(2,\R)$ and if  $\pm c+1<0$ it corresponds to $SU(2)$.  This can be checked computing the Killing form of the Lie algebra and showing that it is non-degenerate and, moreover, negative exactly when  $\pm c+1<0$.  Observe also that the case 
$$
 \left ( \begin{array}{cc}

                        1     	&    -1         	\\

                        c	&      -1	      \\
                        
                       0 	&      0	

                \end{array} \right ),
$$
with $c<0$ is equivalent, using a reordering of the vectors $X_1$ and $X_2$, to the case
$$
\left ( \begin{array}{cc}

                        1     	&    1         	\\

                        -c	&      -1	      \\
                        
                       0 	&      0	

                \end{array} \right ).
$$

 If $c=\mp 1$ then it is solvable.  We obtain the Euclidean and Poincar\'e groups.
 Indeed, for $c=-1$ we get (making $e_1=X_1-X_2, \ e_2=Y, \ e_3=X_2$)
 the usual relations $[e_1,e_2]=0,\ [e_1,e_3]=-e_2, \ [e_2,e_3]=e_1$ of the Euclidean group.
 For $c=1$ we get (making $e_1=X_1+X_2, \ e_2=Y, \ e_3=X_2$) the usual relations $[e_1,e_2]=0,\ [e_1,e_3]=e_2, \ [e_2,e_3]=e_1$ of the Poincar\'e group.
\item
If $a= 0$ and $c\neq 0$, we normalize to
$$
 \left ( \begin{array}{cc}

                        0     	&    -1       	\\

                       1	&      0	      \\
                        
                       0 	&      0	

                \end{array} \right )
$$
which corresponds to the Lie algebra of $SU(2)$,
or
$$
 \left ( \begin{array}{cc}

                        0     	&    1       	\\

                       1	&      0	      \\
                        
                       0 	&      0	

                \end{array} \right ),
$$
which corresponds to the Lie algebra of $\SL(2,\R)$ (make $e_1=X_1-X_2, \ e_2=X_1+X_2, \ e_3=Y$ to obtain the usual relations $[e_1,e_2]=2e_3,\ [e_3,e_1]=-e_1, \ [e_3,e_2]=e_2$).
\item
If $a=c=0$ we may write
$$
\left ( \begin{array}{cc}

                        0     	&    \pm 1         	\\

                        0	&      0	      \\
                        
                       0 	&      0	

                \end{array} \right ).
$$
This corresponds to two solvable groups. One of them (the case the upper right coefficient is $+1$)  is the euclidian group (Bianchi VII) which we can recognise making $e_1=Y, e_2=X_1, e_3=X_2$ so the algebra becomes 
$[e_1,e_2]=0,\ [e_3,e_1]=e_2, \ [e_3,e_2]=-e_1$.  The other case corresponds to the Poincar\'e group (Bianchi VI).
\end{enumerate}
\end{itemize}

\subsubsection{$\ad_Y\mathfrak{p}\nsubseteq\mathfrak{p}$}

Then $a_{31}$ and $a_{32}$ do not vanish at the same time. One can suppose that $a_{31}\neq0$ by exchanging the vectors $X_1$ and $X_2$.  Observe now that from
 $$
a_{31}a_{12}-a_{32}a_{11}=0
$$
$$
a_{31}a_{22}-a_{32}a_{21}=0.
$$
and $a_{31}\neq0$   we have $a_{11}a_{22}-a_{12}a_{21}=0$.
We normalize the matrix of $\ad_Y$ as
$$
 \left ( \begin{array}{cc}

                        a     	&    b       	\\

                        c	&      -a	      \\
                        
                       1	&      f	

                \end{array} \right ),
$$
where $a^2+bc=0$, $af-b=0$ and $cf+a=0$.  We may write then
$$
 \left ( \begin{array}{cc}

                        -cf     	&    -cf^2       	\\

                        c	&      cf	      \\
                        
                       1	&      f	

                \end{array} \right ).
$$
 We normalize further:
\begin{enumerate}
\item
If $f\neq 0$ and $c=0$ we may normalize to
$$
 \left ( \begin{array}{cc}

                        0    	&    0      	\\

                      0 	&      0	      \\
                        
                       1	&      1	

                \end{array} \right )
$$
This corresponds to the solvable group obtained by the direct product of the affine group with $\R$: An invariant path structure on the decomposable solvable group (the only group with derived algebra of dimension one which is not central).
\item
If $f=0$  and $c\neq 0$, we may normalize to
$$
 \left ( \begin{array}{cc}

                        0     	&    0      	\\

                       c 	&      0	      \\
                        
                       1	&      0	

                \end{array} \right )
$$
which corresponds to a family of solvable groups (make $e_1=Y,e_2=cX_2+Y,e_3=X_1$,  then we obtain  $[e_1,e_2]=0,\ [e_1,e_3]=e_2, \ [e_2,e_3]=cY+cX_2+Y=e_2+ce_1$.  The two 
families of solvable groups appear: Bianchi VI (for $c>-1/4$ and $c\neq 0$) and VII  (for $c<-1/4$).  Moreover if $c=-1/4$ then we get Bianchi IV.
\item If $f=0$  and $c= 0$ we obtain another invariant path structure on the decomposable solvable group.  The matrix $ad_Y$ is normalized to
$$
 \left ( \begin{array}{cc}

                        0     	&   0     	\\

                       0 	&      0	      \\
                        
                       1	&      0	

                \end{array} \right ).
$$
\item If $f\neq 0$ and $c\neq 0$ and we may normalize to
$$
 \left ( \begin{array}{cc}

                         -c     	&    -c      	\\

                        c	&      c	      \\
                        
                       1	&      1

                \end{array} \right ),
$$
and a computation shows that for $c>-1/4, c\neq 0$ the group is of type Bianchi VI and for $c<-1/4$ the group is of type  VII.  In the case $c=-1/4$ we obtain type IV.

\end{enumerate}

\subsection{Path structure curvatures of the homogeneous examples}\label{section:curvatures}

The path geometry of the homogenous examples may be described using strict path structure invariants.  Indeed, we showed that there exists parallelisms for each homogeneous  three dimensional path structure. In particular there exists strict path structures associated to each of those homogeneous three dimensional path structures.  In many cases the parallelism and therefore the strict structures are canonical.

Fix a parallelism $X_1,X_2,Y=-[X_1,X_2]$ of the strict path structure,
defining an embedding $j:M\to Y^1$ into the strict
path structure bundle.
Recall (see the beginning of section \ref{subsection:3dim}) that
the matrix of $\ad_Y$ in this basis is:
$$
A= \left ( \begin{array}{cc}

                        a      	&    b          	\\

                        c 	&      -a	      \\
                        
                       e 	&      f	

                \end{array} \right )
$$
and satisfies
$
af-be=0
$
and
$
cf+ae=0.
$

The goal of this section is to compute the path structure curvatures for each homogeneous structure. 

\begin{prop} \label{proposition:curvatures}
The  curvatures of the invariant
path structure defined by $(\R X_1,\R X_2)$
are given by
$$
j^*\iota ^*Q^1=
-a(\frac{3}{2}b-\frac{1}{3}f^2)
$$
and
$$
j^*\iota ^*Q^2=
-a(\frac{3}{2}c+\frac{1}{3}e^2).
$$

\end{prop}
The embedding $\iota:Y^1\to Y$ of the strict path structures
bundle into the path structure bundle was introduced
in Proposition \ref{proposition:embedding}.

\Pf

We obtain a basis of the Lie algebra corresponding to an invariant structure with dual left invariant forms $\theta, \theta^1,\theta^2$ satisfying equations
\ref{equations:homogeneous}
\begin{equation}\begin{array}{lcl}
d\theta^1&=&a\theta^1\wedge\theta+b\theta^2\wedge\theta\\
d\theta^2&=& c\theta^1\wedge\theta-a\theta^2\wedge \theta\\
d\theta &=&e\theta^1\wedge\theta+f\theta^2\wedge \theta+\theta^1\wedge\theta^2
\end{array}
\end{equation}

The goal now is to identify the strict path geometry invariants of the structures.  For that sake we need to obtain the structure equations \ref{equations:strict}.  Note that $\theta$ is the fixed contact form.

First, in order to obtain equation $d\theta=\theta^1\wedge \theta^2$ write  $\theta^1$ and $\theta^2$ in terms of a new basis $\theta, \theta^1-f\theta$ and $\theta^2+e\theta$.  We obtain then, using the same symbols  $\theta, \theta^1,\theta^2$ for the new basis, the equations
\begin{equation}\begin{array}{lcl}
d\theta^1&=&a\theta^1\wedge\theta+b\theta^2\wedge\theta-f\theta^1\wedge\theta^2\\
d\theta^2&=& c\theta^1\wedge\theta-a\theta^2\wedge \theta+e\theta^1\wedge\theta^2\\
d\theta &= &\theta^1\wedge\theta^2
\end{array}
\end{equation}

Now we proceed as in section \ref{section:pseudo} (equations \ref{equations:jembedding}) and write the connection and torsion forms imposing the following equations
\begin{equation}
\begin{array}{lcl}
d\theta^1 -3\upsilon\wedge \theta^1 & =&\theta\wedge \tau^1,\\
d\theta^2 +3\upsilon\wedge \theta^2&=&\theta\wedge \tau^2,\\
d\theta&=&\theta^1\wedge \theta^2.
\end{array}
\end{equation}
where we write the pull-back of the forms  $\upsilon, \tau^1$ and $\tau^2$ by the embedding $j$ using the same letters. 
We obtain
\begin{equation}\begin{array}{lcl}
d\theta^1&=&\theta^1\wedge (a\theta-f\theta^2) -b\theta\wedge\theta^2\\
d\theta^2&=& -\theta^2\wedge( a\theta+e\theta^1) -c\theta\wedge\theta^1\\
d\theta &= &\theta^1\wedge\theta^2
\end{array}
\end{equation}

Comparing with the structure equations we obtain 
\begin{equation}\begin{array}{lcl}
\upsilon &= &-\frac{1}{3}(a\theta+e\theta^1-f\theta^2) \\
\tau^1 & =&  -b\theta^2\\
\tau^2 & = & - c\theta^1
\end{array}
\end{equation}

We compute equation \ref{dtheta11},  the exterior derivative of $\upsilon$,
$
d \upsilon=R\theta^1\wedge \theta^2+W^1\theta^1\wedge\theta+W^2\theta^2\wedge\theta
$ 
and obtain

\begin{equation}
\begin{array}{lcl}
R &= & -\frac{1}{3}(a-2ef) \\
W^1 & =&   -\frac{1}{3}(ae-cf)=\frac{2}{3}cf\\
W^2 & = & -\frac{1}{3}(be+af)= -\frac{2}{3}be.
\end{array}
\end{equation}

From
$$
d\tau^1+3\tau^1\wedge \upsilon=3W^2\theta^1\wedge \theta^2+S^1_1\theta\wedge\theta^1+S^1_2\theta\wedge\theta^2
$$
we obtain
$$
-bd\theta^2+3b\theta^2\wedge \frac{1}{3}(a\theta+e\theta^1)=3W^2\theta^1\wedge \theta^2+S^1_1\theta\wedge\theta^1+S^1_2\theta\wedge\theta^2.
$$
That is
$$
-b( -\theta^2\wedge( a\theta+e\theta^1) -c\theta\wedge\theta^1)+3b\theta^2\wedge \frac{1}{3}(a\theta+e\theta^1)=3W^2\theta^1\wedge \theta^2+S^1_1\theta\wedge\theta^1+S^1_2\theta\wedge\theta^2
$$
which implies
$$
 W^2= -\frac{2}{3}eb,\ \  S^1_1=bc, \ \ \ S^1_2=-2ab.
$$
Analogously, from
$
d\tau^2-3\tau^2\wedge \upsilon=3W^1\theta^1\wedge \theta^2+S^2_1\theta\wedge\theta^1+S^2_2\theta\wedge\theta^2,
$
we obtain 
$$
-c(\theta^1\wedge (a\theta-f\theta^2) -b\theta\wedge\theta^2) +3(- c\theta^1)\wedge \frac{1}{3}(a\theta-f\theta^2)
=3W^1\theta^1\wedge \theta^2+S^2_1\theta\wedge\theta^1+S^2_2\theta\wedge\theta^2,
$$
which implies
$$
W^1=\frac{2}{3}cf, \ \ \ S^2_1=2ac, \ \ \ S^2_2=bc.
$$
Moreover, we have the relation
$$
S^1_1=S^2_2=\tau^1_2\tau^2_1.
$$
Remark that from equations \ref{dW1} and \ref{dW2}
 
 \begin{equation}
 dW^1+3W^1\upsilon=W^1_0\theta+W^1_1\theta^1+W^1_2\theta^2
 \end{equation}
 and
 \begin{equation}
  dW^2-3W^2\upsilon=W^2_0\theta+W^2_1\theta^1+W^2_2\theta^2,
 \end{equation}
 we obtain

$$
-\frac{2}{9}cf (a\theta+e\theta^1-f\theta^2)
$$
$$
W^1_0=-\frac{2}{3}cfa,\ \ W^1_1=-\frac{2}{3}cfe,\ \ W^1_2=\frac{2}{3}cf^2
$$
and 
$$
\frac{2}{9}eb(a\theta+e\theta^1-f\theta^2)
$$

$$
W^2_0=-\frac{2}{3}eba,\ \ W^2_1=-\frac{2}{3}e^2b,\ \ W^2_2=\frac{2}{3}ebf
$$

Substituting the expressions obtained above in Proposition \ref{proposition:embedding} completes the proof.

\EPf

\subsection{Tables of invariant path structures on three dimensional Lie groups}\label{subsection:table}

In this section we put together the classification of path structures on the three dimensional Lie groups in the form of two tables.  One for solvable Lie groups and the other for the remaining simple groups.  In the following, we use the Bianchi classification of three dimensional Lie groups. Recall that the group Bianchi VI$_0$ is the Poincar\'e group and Bianchi VII$_0$ is the Euclidean group. The Bianchi groups of type I and V don't have invariant contact structures and therefore don't appear in the tables.  The proof of the classification can be read in the previous section \ref{subsection:3dim}.

\newpage
\begin{tabular}{ |p{5cm}||p{5cm}|p{2cm}|p{2cm}|  }
 \hline
 \multicolumn{4}{|c|}{Path structures on three dimensional solvable Lie groups} \\
 \hline
  & $\ad_Y$  & $Q^1$ & $Q^2$ \\
 \hline 
{{Heisenberg}, Bianchi II} &  {\small{$\left ( \begin{array}{cc}

                        0     	&    0         	\\

                        0	&      0	      \\
                        
                       0 	&      0	

                \end{array} \right ) $}}  &	$0$	& $0$  \\
\hline
\multirow{2}{15em}{$Aff(2)\times \R$  \\ Bianchi III} & $\left ( \begin{array}{cc}

                        0     	&    0         	\\

                        0	&     0	      \\
                        
                       1 	&      0	

                \end{array} \right ) $ &   $0$	& $0$\\ 
              
& $\left ( \begin{array}{cc}

                        0     	&    0         	\\

                        0	&     0	      \\
                        
                       1 	&      1	

                \end{array} \right ) $  & $0$  & 0\\ 
  \hline
  
 Bianchi IV 	&  
 $
 \left ( \begin{array}{cc}

                        0     		&    0      	\\

                       -\frac{1}{4} 	&      0	      \\
                        
                       1			&      0	

                \end{array} \right )
$  	&   $0$	& 0\\ 
              &	
$
 \left ( \begin{array}{cc}

                        \frac{1}{4}     		&    \frac{1}{4}      	\\

                       -\frac{1}{4} 	&     -\frac{1}{4}	      \\
                        
                       1			&      1	

                \end{array} \right )
$  &	 $-\frac{1}{4}(\frac{3}{8}-\frac{1}{3})$ & $-\frac{1}{4}(-\frac{3}{8}+\frac{1}{3})$  	\\
  \hline
\multirow{4}{15em}{Bianchi VI} &
$
 \left ( \begin{array}{cc}

                        0     	&   - 1     	\\

                       0 	&      0	      \\
                        
                       0		&      0	

                \end{array} \right )
$ {\tiny{Bianchi VI$_0$}}&		$0$ &   0\\
&
$
 \left ( \begin{array}{cc}

                        1     	&   - 1     	\\

                       1 	&      -1	      \\
                        
                       0		&      0	

                \end{array} \right )
$ {\tiny{Bianchi VI$_0$}}&	 $\frac{3}{2}$ & $-\frac{3}{2}$	   \\
 &  	$
 \left ( \begin{array}{cc}

                        0     	&    0      	\\

                       c 	&      0	      \\
                        
                       1		&      0	

                \end{array} \right )
$, {\tiny{$0\neq c>-\frac{1}{4}$}}
&  	$0$	&   0	\\	
&
$
 \left ( \begin{array}{cc}

                        -c     	&    -c     	\\

                       c 	&      c	      \\
                        
                       1		&      1	

                \end{array} \right )
$, {\tiny{$0\neq c>-\frac{1}{4}$}}
&	$c(-\frac{3}{2}c-\frac{1}{3})$ & $c(\frac{3}{2}c+\frac{1}{3})$	   	\\
  \hline
\multirow{2}{15em}{Bianchi VII} & 
$
 \left ( \begin{array}{cc}

                        0     	&    1     	\\

                       0 	&      0	      \\
                        
                       0 	&      0	

                \end{array} \right )
$  {\tiny{Bianchi VII$_0$}}
&	$0$	&  	0 \\
&  
$
 \left ( \begin{array}{cc}

                        1     	&    1      	\\

                       -1 	&      -1	      \\
                        
                       0	&      0	

                \end{array} \right )
$  {\tiny{Bianchi VII$_0$}}
&	$-\frac{3}{2}$ & $\frac{3}{2}$ 	 \\

&  
$
 \left ( \begin{array}{cc}

                        0     	&    0      	\\

                       c 	&      0	      \\
                        
                       1		&      0	

                \end{array} \right )
$, {\tiny{$c<-\frac{1}{4}$}}
&	$0$	&  	0 \\
&
$
 \left ( \begin{array}{cc}

                        -c     	&    -c     	\\

                       c 	&      c	      \\
                        
                       1		&      1	

                \end{array} \right )
$, {\tiny{$c<-\frac{1}{4}$}}
&		$c(-\frac{3}{2}c-\frac{1}{3})$ & $c(\frac{3}{2}c+\frac{1}{3})$ 	\\

 \hline 
\end{tabular}

\begin{tabular}{ |p{5cm}||p{4cm}|p{1.5cm}|p{1.5cm}|  }
 \hline
 \multicolumn{4}{|c|}{Path structures on three dimensional simple Lie groups} \\
 \hline
  & $\ad_Y$ & $Q^1$ &$Q^2$\\
 \hline 
 
\multirow{3}{15em}{$\SL(2,\R)$ \\ Bianchi VIII} 	
&
$ \left ( \begin{array}{cc}

                        1     	&    0        	\\

                        0	&      -1	      \\
                        
                       0 	&      0	

                \end{array} \right )$ 	 &	0 & 0\\
&
$ \left ( \begin{array}{cc}

                        0     	&    1         	\\

                        1	&      0	      \\
                        
                       0 	&      0	

                \end{array} \right )$ 	& 0	& 0\\
      & 		$ \left ( \begin{array}{cc}

                        1     	&    1         	\\

                        c	&      -1	      \\
                        
                       0 	&      0	

                \end{array} \right )$, {\tiny{$c>-1$}} 	&	
                	$-\frac{3}{2}$ & $-\frac{3}{2}c$\\
& $ \left ( \begin{array}{cc}

                        1     	&    -1         	\\

                        c	&      -1	      \\
                        
                       0 	&      0	

                \end{array} \right )$, {\tiny{$0<c<1$}} 	&		$\frac{3}{2}$ & $-\frac{3}{2}c$\\
    \hline

\multirow{2}{15em}{$SU(2)$  \\ Bianchi IX}
&
$ \left ( \begin{array}{cc}

                        0     	&    -1        	\\

                        1	&      0	      \\
                        
                       0 	&      0	

                \end{array} \right )$ 	&	0 & 0\\

                & $ \left ( \begin{array}{cc}

                        1     	&    1         	\\

                        c	&      -1	      \\
                        
                       0 	&      0	

                \end{array} \right )$, {\tiny{$c<-1$ }}	&		$-\frac{3}{2}$ & $-\frac{3}{2}c$\\
& $ \left ( \begin{array}{cc}

                        1     	&    -1         	\\

                        c	&      -1	      \\
                        
                       0 	&      0	

                \end{array} \right )$, {\tiny{$c>1$}} 	&		$\frac{3}{2}$ & $-\frac{3}{2}c$\\
    \hline

 \hline 
\end{tabular}

\subsection{Natural invariants in the case of
$\SL(2,\R)$}\label{subsection-invariantsSL2}
In this subsection, we relate the left-invariant
structures on $\SL(2,\R)$ to the geometry of the Killing form of ${\mathfrak {sl}}(2,\R)$ which defines a Lorentzian structure on $\SL(2,\R)$. 

 Recall the parallelism $X_1,X_2,Y=-[X_1,X_2]$ and the matrix of $\ad_Y$ in this basis (section \ref{subsection:3dim}):
$$
A= \left ( \begin{array}{cc}

                        a      	&    b          	\\

                        c 	&      -a	      \\
                        
                       e 	&      f	

                \end{array} \right )
$$
satisfying
$
af-be=0
$
and
$
cf+ae=0.
$
We have
$$
\ad_{X_1}=\left ( \begin{array}{ccc}

                        0      	&    0      &   -a	\\

                        0 	&      0	&   -c   \\
                        
                       0 	&      -1	& -e

                \end{array} \right ), \ \ \ad_{X_2}=\left ( \begin{array}{ccc}

                        0     	&    0      &   -b	\\

                        0 	&      0	&   a   \\
                        
                       1 	&      0	& -f

                \end{array} \right ),  \ \ \ad_{Y}=\left ( \begin{array}{ccc}

                        a     	&    b      &   0	\\

                        c 	&      -a	&   0   \\
                        
                       e 	&      f	& 0

                \end{array} \right ).
$$
Therefore, the Killing form of the Lie algebra satisfies
$$
\kappa(X_1,X_1)=\tr X_1^2=2c+e^2,\ \ \kappa(X_2,X_2)=\tr X_2^2=-2b+f^2, \ \ \kappa(Y,Y)=\tr Y^2=2(a^2+bc).
$$
$$
\kappa(X_1,X_2)=-2a+ef,\ \ \kappa(X_1,Y)=-ae-cf, \ \ \kappa(X_2,Y)=-be+af.
$$

\subsubsection{Lorentzian geometry of ${\mathfrak {sl}}(2,\R)$}
\label{subsubsection-Lorentziangeometrysl2}
Let us consider the Killing form of ${\mathfrak {sl}}(2,\R)$, $\kappa(u,v)=\frac{1}{2}\tr (uv)$ whose associated quadratic form is 
\begin{equation}\label{equation-definitionqsl2}
 q(u)=-\mathrm{det}(u).
\end{equation}

Then, using the standard basis
\begin{equation}\label{equation-definitionEFH}
 E=\begin{pmatrix}
    0 & 1 \\
    0 & 0
   \end{pmatrix},
F=\begin{pmatrix}
   0 & 0 \\
   1 & 0
  \end{pmatrix},
H=\begin{pmatrix}
   1 & 0 \\
   0 & -1
  \end{pmatrix}
\end{equation}
of ${\mathfrak {sl}}(2,\R)$
satisfying the relation $H=[E,F]$,
$(H,E+F,E-F)$ is an orthonormal basis of signature $(1,1,-1)$ for $q$.
In particular the choice of this basis identifies
${\mathfrak {sl}}(2,\R)$ to the Lorentzian Minkowski space $\R^{1,2}$.
Elements of ${\mathfrak {sl}}(2,\R)$ are thus distinguished
according to the sign of $q$:
$u\in{\mathfrak {sl}}(2,\R)$ is called
\begin{itemize}
 \item \emph{timelike} if $q(u)<0$,
 \item \emph{lightlike} if $q(u)=0$,
 \item \emph{spacelike} if $q(u)>0$.
\end{itemize}
Recall that an element of ${\mathfrak {sl}}(2,\R)$ is called
\emph{hyperbolic} (respectively \emph{parabolic}, \emph{elliptic})
if it generates a one-parameter group of the given kind
in $\SL(2,\R)$.
We have then the following straightforward
correspondence between the geometry of $q$ and the algebraic types
in ${\mathfrak {sl}}(2,\R)$.
\begin{lemma}\label{lemme-interpretationsousgroupesSL2}
Let $u\in{\mathfrak {sl}}(2,\R)$. Then:
 \begin{itemize}
  \item $u$ is timelike $\Leftrightarrow$
  $u$ is elliptic;
  \item $u$ is lightlike $\Leftrightarrow$
  $u$ is parabolic;
  \item $u$ is spacelike $\Leftrightarrow$
  $u$ is hyperbolic.
 \end{itemize}
\end{lemma}

Lastly, $q$ is invariant by the adjoint action of
$\SL(2,\R)$, and more precisely
\begin{equation}\label{equation-AdSL2SO12}
 \mathrm{Ad}\colon g\in \SL(2,\R)\mapsto
 Ad_g\in SO^0(q)
\end{equation}
is surjective and has kernel $\{\pm id\}$,
$SO^0(q)$ denoting the connected component of the identity in
the stabilizer of $q$ in $\mathrm{GL}(\mathfrak{sl}(2,\R))$.

\subsubsection{The moduli space of left-invariant path structures on
$\SL(2,\R)$}
\label{subsubsection-modulispaceSL2}

\begin{dfn}
We define $\mathcal{S}$ to be the set of left-invariant path structures on
$\SL(2,\R)$.
\end{dfn}
Our goal is now to extract the
natural geometric invariants distinguishing,
in the space $\mathcal{S}$, those which are locally
isomorphic.
\par We say that a plane $P$ in $\mathfrak{sl}(2,\R)$
is
\begin{itemize}
 \item \emph{timelike} if $q|_P$ has Lorentzian signature $(1,-1)$,
 \emph{i.e.} if $P$ contains a timelike line;
 \item \emph{lightlike} if $q|_P$ has signature $(1,0)$,
 \emph{i.e.} if it is degenerated,
 or if $P$ contains exactly one lightlike line;
 \item \emph{spacelike} if $q|_P$ has euclidean signature $(1,1)$,
 \emph{i.e.} if $P$ contains no lightlike line.
\end{itemize}

Each choice of a plane in $\mathfrak{sl}(2,\R)$ defines, by translation,  a distribution on $\SL(2,\R)$.
Note that the adjoint action by
$\SL(2,\R)$ on the set of planes in $\mathfrak{sl}(2,\R)$ has three orbits.
Indeed since
this action factorizes through the action of $SO^0(1,2)$ on the
Minkowski space $\R^{1,2}$
according to \eqref{equation-AdSL2SO12},
it is sufficient to check that $SO^0(1,2)$
has three orbits on the set of planes of $\R^{1,2}$,
which the following dichotomy shows:
\begin{description}
 \item[Timelike planes] The restriction of $SO^0(1,2)$
 to the
open set of timelike lines of $\R^{1,2}$
is conjugated to the
action of the group of orientation-preserving isometries
on the hyperbolic plane,
which is transitive
with the stabilizer of any point $p$
acting transitively on unitary tangent vectors at $p$.
This shows in particular that $SO^0(1,2)$ acts transitively
on the set of timelike planes.
\item[Spacelike planes] The restriction of $SO^0(1,2)$
 to the
open set of spacelike half-lines of $\R^{1,2}$
is conjugated to the
action of the group of orientation and time-orientation
preserving isometries
on the de-Sitter space,
which is transitive
with the stabilizer of any point $p$
acting transitively on unitary spacelike tangent vectors at $p$.
This shows in particular
that $SO^0(1,2)$ acts transitively
on the set of spacelike planes.
\item[Lightlike planes] $SO^0(1,2)$
 acts transitively on the
closed set of lightlike lines of $\R^{1,2}$,
and any lightlike line is contained in a unique lightlike plane.
This shows that $SO^0(1,2)$ acts transitively
on the set of lightlike planes.
\end{description}
The open orbits of timelike and spacelike planes
of $\mathfrak{sl}(2,\R)$
correspond to the two invariant contact distributions in $\SL(2,\R)$. The closed orbit of timelike planes
corresponds to an integrable distribution which will not be relevant in the sequel.

Consider the projective space $\mathbb{P}(\mathfrak{sl}(2,\R))$. We introduce the two open subsets in 
$\mathbb{P}(\mathfrak{sl}(2,\R))\times \mathbb{P}(\mathfrak{sl}(2,\R))$,
\begin{align}\label{equation-Espaceettime}
 \mathcal{E}_{space}&=
 \{(D_1,D_2)\}|\{D_1\neq D_2, D_1\oplus D_2\text{~spacelike}\}, \\
 \mathcal{E}_{time}&=
 \{(D_1,D_2)\}|\{D_1\neq D_2, D_1\oplus D_2\text{~timelike}\}
\end{align}
 as parameter spaces for
$\mathcal{S}$.
Note that two pairs of distinct points of
$\mathbb{P}(\mathfrak{sl}(2,\R))$ are in the same orbit
under the adjoint action of $\SL(2,\R)$,
if and only if the pairs of left-invariant line fields
that they generate are related by an element of $\SL(2,\R)$.
Hence any invariant of adjoint orbits of $\SL(2,\R)$
will be relevant for our study.
There exists one natural   invariant in each
of the open subsets $\mathcal{E}_{space}$ and $\mathcal{E}_{time}$. 

\begin{dfn} For $(D_1,D_2)\in\mathcal{E}_{space}$ or $\mathcal{E}_{time}$,
pairs with a lightlike line being excluded, define the cross-ratio
$$
cr(D_1,D_2)=\frac{\kappa (X_1,X_2)\kappa (X_2,X_1)}{\kappa (X_1,X_1)\kappa (X_2,X_2)}
$$
with $X_i$ a generator of $D_i$ for $i=1,2$.
\end{dfn}

\par For any $u\subset\mathfrak{sl}(2,\R)$
$\tilde{u}$ denotes the left-invariant vector field of
$\SL(2,\R)$ generated by $u$,
with an analog notation for left-invariant line or plane fields.

\begin{prop}\label{proposition-parametrisation}
The map
\begin{equation}\label{equation-parametrisationS}
 (D_1,D_2)\in\mathcal{E}_{space}\cup\mathcal{E}_{time}\mapsto
 (\widetilde{D_1},\widetilde{D_2})\in\mathcal{S}
\end{equation}
is surjective.
Moreover, $(\widetilde{D_1},\widetilde{D_2})$ and
$(\widetilde{D_1'},\widetilde{D_2'})$ are locally isomorphic if,
and only if one of the following mutually
exclusive conditions is satisfied,
with $P=D_1\oplus D_2$ and $P'=D_1'\oplus D_2'$.
\begin{enumerate}
 \item All the lines $D_i$ and $D_i'$ are lightlike,
 and in this case the associated path structures are flat.
 \item Exactly one of the lines $(D_1,D_2)$,
 respectively $(D_1',D_2')$ is lightlike and the others are of the same type.  In this case one of the Cartan curvatures is null.
 \item $P$ and $P'$
 are spacelike
 and $cr(D_1,D_2)=cr(D'_1,D'_2)$.
 \item $P$ and $P'$ are timelike,  $D_i$ of the same type as $D_i'$ for $1\leq i\leq 2$,
 none of the lines is lightlike,
 and $cr(D_1,D_2)
 =cr(D_1',D_2')$.  There are three components corresponding to $D_1$ and $D_2$ both timelike, both space-like or  of different type.
 
\end{enumerate}
\end{prop}
\begin{proof}
Surjectivity follows from the fact that contact distributions arise  either from  timelike planes or spacelike planes.  The closed orbit of lightlike planes is excluded.

Consider vectors $X_1$ and $X_2$ generating $D_1$ and $D_2$.  Let $X_1,X_2,Y=-[X_1,X_2]$ be a basis of the Lie algebra and the matrix of $\ad_Y$ in this basis (section \ref{subsection:3dim})
for $\mathfrak{sl}(2,\R)$ be written as:
$$
A= \left ( \begin{array}{cc}

                        a      	&    b          	\\

                        c 	&      -a	      \\
                        
                       0 	&      0	

                \end{array} \right ).
$$
 From the computation of the Killing form we obtain $\kappa(X_1,X_1)=2c$, $\kappa(X_2,X_2)=-2b$ and $\kappa(X_1,X_2)=-2a$. Also $Q^1=-\frac{3}{2}ab$ and 
$Q^2=-\frac{3}{2}ac$.  One also have
$$
cr(D_1,D_2)=\frac{\kappa (X_1,X_2)\kappa (X_2,X_1)}{\kappa (X_1,X_1)\kappa (X_2,X_2)}=-\frac{a^2}{bc}.
$$
The proposition now follows from the following computations:
\begin{itemize}
\item If $D_1$ and $D_2$ are lightlike, the invariant flag structure is flat.
\item If the two lines are orthogonal, one timelike and the other spacelike, the invariant flag structure is flat.
\item  If $c=0, b>0$, $(D_1,D_2)$ is timelike ($X_1$ lightlike and $X_2$ timelike)  and one can normalize so that $a=1, b=1$.  One verifies that $Q^1=-1$ and $Q^2=0$.
\item If $c=0, b<0$, $(D_1,D_2)$ is timelike ($X_1$ lightlike and $X_2$ spacelike)  and one can normalize so that $a=1, b=-1$. One verifies that $Q^1=1$ and $Q^2=0$.
\item If $c>0, b<0$, $(D_1,D_2)$ is spacelike (if $a^2+bc<0$) or timelike (if $a^2+bc>0$). One can normalize so that $a=1, b=-1$.  Then $Q^1=\frac{3}{2}$, $Q^2=-\frac{3}{2}c$ and $cr(D_1,D_2)=1/c$. 
\item  If $c>0, b>0$, $(D_1,D_2)$ is timelike ($X_1$ spacelike and $X_2$ timelike)  and one can normalize so that $a=1, b=1$. Then $Q^1=-\frac{3}{2}$, $Q^2=-\frac{3}{2}c$ and $cr(D_1,D_2)=-1/c<0$.
\item  If $c<0, b>0$, $(D_1,D_2)$ is timelike ($X_1$ and $X_2$ timelike)  and one can normalize so that $a=1, b=1$.
Then $Q^1=-\frac{3}{2}$, $Q^2=-\frac{3}{2}c$ and $cr(D_1,D_2)=-1/c>0$.
\end{itemize}
To conclude the reverse implication of the Proposition,
we use the fact that two homogeneous path structures having the same invariants $Q^i$
are locally isomorphic, as a consequence of the classification
of section \ref{subsection:3dim}
and of the proof of Theorem \ref{thm-reduction}.
\end{proof}

\section{Appendix}

 \subsection{Transformation properties of the adapted connection for path structures}\label{A-transformation}
  
  In this section we review the explicit transformation properties of the Cartan connection of  the $B_\R$ (Borel group) principal bundle associated to a path structure (see \cite{FVflag} for an analogous computation in the context of a $B_\C$ principal bundle).
  
We  compute  $Ad_{h^{-1}}\pi$ for an element 
$$
h=
\left( \begin{array}{ccc}

                        a  	&    	c      	&       e \\

                       0				 &     \frac{1}{ab}		&	f	\\
                       0   				&  	0 				&		b
                      
       \end{array} \right )
$$
We have, for a constant $h$,
$$
\tilde{\pi}=Ad_{h^{-1}}\pi.
$$

A computation shows that

\begin{align}
\label{adjunta}
\tilde\omega  &= \frac{a}{b}\, \omega\notag\\
\tilde{\omega}^1 &=a^2b\, \omega^1-a^2f\,\omega\notag\\
\tilde\omega^2 &=\frac{1}{ab^2}\,\omega^2+\frac{c}{b}\,\omega\notag\\
\tilde\phi &=\phi-\frac 1 2 abc\,\omega^1-\frac{f}{2b}\,\omega^2+(\frac 1 2 acf-\frac{e}{b})\,\omega\\
\tilde w &=w -\frac{1}{2}abc\,\omega^1+\frac{f}{2b}\,\omega^2+\frac 1 2acf\,\omega\notag\\
\tilde\phi^1 &=b^2a\, \phi^1-3abf\, w+baf\,\phi+bae\, \omega^1-f^2a\,\omega^2-fae\, \omega\notag\\
\tilde\phi^2 &=\frac{1}{ba^2}\, \phi^2+\frac{3c}{a}\, w+\frac{c}{a}\,\phi-bc^2\, 
\omega^1+(-\frac{
e}{a^2b^2}+\frac{cf}{ab})\,\omega^2+(-\frac{ce}{ab}+{fc^2})\, \omega\notag\\
\tilde\psi &=\frac{b}{a}\,\psi+(\frac{2e}{a}-bcf)\,\phi -bce\,\omega^1+(-\frac{fe}{ab}+{cf^2})\,
\omega^2-cb^2\,\phi^1+\frac{f}{a}\, \phi^2+3fbc\, w+(-\frac{e^2}{ab}+{fce})\, \omega\notag
\end{align}
and for the curvature 
$$
\Pi= d\pi+\pi\wedge\pi=\left ( \begin{array}{ccc}
                        0 &   \Phi^{2}      &  \Psi   \\

                         0  &  0    &  \Phi^1  \\

                           0  &   0  &    0
               \end{array} \right ),
$$
$$
\tilde \Pi=Ad_{h^{-1}}\Pi=
\left ( \begin{array}{ccc}
                        0 &  \frac{1}{a^2b}\Phi^{2}      & \frac{b}{a}\Psi+ \frac{f}{a}\Phi^{2} -cb^2 \Phi^1\\

                         0  &  0    &  ab^2 \Phi^1 \\

                           0  &   0  &    0
               \end{array} \right ).
$$

That is, 
$$
\tilde \Phi^1=ab^2\Phi^1,
$$
$$
\tilde \Phi^2=\frac{1}{a^2b}\Phi^2,
$$
and
$$
 \tilde \Psi =\frac{b}{a}\Psi +\frac{f}{a}\Phi^2-cb^2\Phi^1.
 $$
Recalling that $
\Phi^1=Q^1\omega\wedge \omega^2,\ \ 
\Phi^2=Q^2\omega\wedge \omega^1 \ \mbox{and}\ \ 
\Psi =\left( U_1\omega^1+ U_2\omega^2\right) \wedge \omega,
$
one may write 
therefore
$$
\tilde Q^1\tilde \omega\wedge \tilde\omega^2=ab^2\, Q^1\omega\wedge\omega^2.
$$
But $\tilde Q^1\tilde \omega\wedge \tilde\omega^2=\tilde Q^1\frac{a}{b}\, \omega\wedge \frac{1}{ab^2}\omega^2$
and then
\begin{align}\label{transformationsQ}
\tilde Q^1=ab^5\, Q^1,\\
\tilde Q^2=\frac{1}{a^5b}\, Q^2
\end{align}
and, analogously,
\begin{align}\label{transformationsU}
\tilde U_1=\frac{b}{a^4}(U_1-\frac{f}{b}Q^2),\\
\tilde U_2=\frac{b^4}{a}( U_2+abcQ^1).
\end{align}

These transformation properties imply that we can define two tensors on $Y$ which are invariant under $H$ and will give rise to two tensors on $M$.
Indeed
$$
 Q^1\,\omega^2\wedge  \omega\otimes { \omega}\otimes   e_1
$$
and
$$
 Q^2\,\omega^1\wedge  \omega\otimes { \omega} \otimes  e_2,
$$
where $e_1$ and $e_2$ are duals to $\omega^1$ and $\omega^2$ in the dual frame of the coframe bundle of $Y$, are easily seen to be $H$-invariant.

\subsection{The embedding $\iota : Y^1\to Y$}\label{section:pseudo-embedding}

We recall here the embedding described in \cite{FVflag,FVODE}. We give additional details of the computation of the path structure curvatures in terms of invariants of the strict path structure. This embedding allows curvatures of path structures to be expressed 
in terms of curvatures of strict path structures which, in turn,  are much easier to compute.

\begin{prop} \label{proposition:embedding}
There exists a unique equivariant embedding of fiber bundles $\iota : Y^1\to Y$ such that 
\begin{equation}\label{iota}
\iota ^* \omega =\theta,\ \ 
\iota ^* \omega^1 =\theta^1,\ \ 
\iota ^* \omega^2 =\theta^2,\ \ 
\iota ^* \phi =0.
\end{equation}
Moreover 
$$
\iota ^*Q^1=S^1_2+\frac{3}{2}R\tau^1_2+2W^2_2-\frac{1}{2}R_{22}
$$
and
$$
\iota ^*Q^2=-S^2_1+\frac{3}{2}R\tau^2_1 -2W^1_1+\frac{1}{2}R_{11}.
$$

\end{prop}

\Pf
In order to express the curvatures of the path structures,  we first compute the pull-back of the connection of the path structure
in terms of the connection of the strict structure.  The pull back of the structure equations are (here we denote the pull-back of a form by the same letter except for the forms $\omega, \omega^1, \omega^2$ which are given, from the above proposition, as   $\theta, \theta^1, \theta^2$):

 \begin{align}  
d\theta 		&= \theta^1\wedge\theta^{2}  \notag\\
d\theta^1	&= 
3w \wedge \theta^1 +\theta\wedge \phi^1\notag\\ 
d\theta^2 	& =  -3 w\wedge \theta^2-\theta\wedge \phi^2\notag\\
\theta \wedge \psi			& = \frac{1}{2}(\phi^{2}\wedge \theta ^1+\phi^1\wedge \theta^{2})\notag\\
dw 			& = \frac{1}{2}(-\phi^{2}\wedge \theta ^1+\phi^1\wedge \theta^{2})\notag\\
Q^1\theta\wedge \theta^2 	&=d\phi^1 +3\phi^1\wedge w+\theta^1\wedge \psi  \notag\\
Q^2\theta\wedge \theta^1 	& = d\phi^2 -3\phi^2\wedge w-\theta^2\wedge \psi \notag\\ 
 d\psi-\phi^1\wedge \phi^2 	&= (U_1\theta^1+ U_2\theta^2)\wedge \theta.\notag
 \end{align}
 
 It follows, by comparing with Proposition \ref{cartan-strict}, that 
 \begin{align}  
w	&=\upsilon +c\theta \notag\\
\phi^1 &= \tau^1-3c\theta^1 +E^1\theta\notag\\
\phi^2 &= -\tau^2-3c\theta^2+ E^2\theta\notag.
 \end{align} 
 where $E^1,E^2$ and $c$ are functions on $Y^1$.  We obtain then that
 $$
 \psi=\frac{1}{2}(E^2\theta^1+E^1\theta^2+G\theta)
 $$
where $G$ is a function on $Y^1$.  Substituting the formulas for $\phi^1, \phi^2$ and $w$ in the equation for $dw$ we obtain, using equation \ref{dtheta11}

$$
dw=R\theta^1\wedge \theta^2+W^1\theta^1\wedge\theta+W^2\theta^2\wedge\theta +dc\wedge \theta + cd\theta= \frac{1}{2}(-6c\theta^1\wedge \theta^2 -E^2\theta\wedge \theta^1+E^1\theta\wedge \theta^2),
$$
which implies, comparing the terms in $\theta^1\wedge \theta^2$,
$$
c=-\frac{R}{4},
$$
and then, using $
 dR=R_0\theta +R_1\theta^1+R_2\theta^2$, we obtain
 $$
 E^1=-2W^2+\frac{R_2}{2},\ \   E^2=2W^1-\frac{R_1}{2}.
 $$
 \begin{itemize}
 \item
 Substituting the formulas for $\phi^1, \psi$ and $w$, with the above values of $E^1$ and $c$, in the equation $Q^1\theta\wedge \theta^2 	=d\phi^1 +3\phi^1\wedge w+\theta^1\wedge \psi$, we obtain
 \begin{equation}\label{Q1strict}
 Q^1\theta\wedge \theta^2 =d(\tau^1-3c\theta^1 +E^1\theta)+3(\tau^1-3c\theta^1 +E^1\theta)\wedge (\upsilon +c\theta)+\theta^1\wedge  \frac{1}{2}(E^1\theta^2+G\theta).
 \end{equation}
  Compute, using equations \ref{dW2} and  \ref{dR2}
 $$
 dE^1\wedge \theta= 
 d(-2W^2+\frac{R_2}{2})\wedge \theta=-2(3W^2\upsilon+W^2_1\theta^1+W^2_2\theta^2)\wedge \theta
 +\frac{1}{2}(3R_2\upsilon-\frac{1}{2}R_0\theta^1+R_{12}\theta^1+R_{22}\theta^2)\wedge \theta,
 $$
 and
 $$
 -3d(c\theta^1)= -3dc\wedge \theta^1-3cd\theta^1=\frac{3}{4}dR\wedge \theta^1+\frac{3}{4}R(-3\theta^1\wedge \upsilon +\theta\wedge \tau^1)
 $$

 Recall also equation \ref{dtau1}
 $$d\tau^1+3\tau^1\wedge \upsilon=3W^2\theta^1\wedge \theta^2+S^1_1\theta\wedge\theta^1+S^1_2\theta\wedge\theta^2.
 $$
 
 Substituting the above expressions into equation \ref{Q1strict} we obtain
 \begin{align}
 Q^1\theta\wedge \theta^2 &=-3\tau^1\wedge \upsilon+3W^2\theta^1\wedge \theta^2+S^1_1\theta\wedge\theta^1+S^1_2\theta\wedge\theta^2 +\frac{3}{4}dR\wedge \theta^1+\frac{3}{4}R(-3\theta^1\wedge \upsilon +\theta\wedge \tau^1)\notag\\
 & 	-2(3W^2\upsilon+W^2_1\theta^1+W^2_2\theta^2)\wedge \theta
 +\frac{1}{2}(3R_2\upsilon-\frac{1}{2}R_0\theta^1+R_{12}\theta^1+R_{22}\theta^2)\wedge \theta
 +	E^1\theta^1\wedge \theta^2\notag\\
 &	+3(\tau^1-3c\theta^1 +E^1\theta)\wedge (\upsilon +c\theta)+\theta^1\wedge  \frac{1}{2}(E^1\theta^2+G\theta),\notag
 \end{align}
 which, using $
 3E^1\theta\wedge \upsilon=(-6W^2+\frac{3R_2}{2})\theta\wedge \upsilon
 $, simplifies to an expression with no terms involving $\upsilon$:
 \begin{align}
 Q^1\theta\wedge \theta^2 &=3W^2\theta^1\wedge \theta^2+S^1_1\theta\wedge\theta^1+S^1_2\theta\wedge\theta^2 +\frac{3}{4}dR\wedge \theta^1+\frac{3}{4}R\theta\wedge \tau^1\notag\\
 & 	-2(W^2_1\theta^1+W^2_2\theta^2)\wedge \theta
 +\frac{1}{2}(-\frac{1}{2}R_0\theta^1+R_{12}\theta^1+R_{22}\theta^2)\wedge \theta
 +	E^1\theta^1\wedge \theta^2\notag\\
 &	+3(\tau^1-3c\theta^1 )\wedge c\theta +\theta^1\wedge  \frac{1}{2}(E^1\theta^2+G\theta).\notag
 \end{align}
 Observe now that the coefficient of $\theta^1\wedge\theta^2$ is
 $$
 3W^2-\frac{3}{4}R_2+\frac{3}{2}E^1= 0.
 $$
Therefore we may simplify to
   \begin{align}
 Q^1\theta\wedge \theta^2 &=(S^1_1+R_0+2W^2_1-\frac{1}{2}R_{12}+\frac{9}{16}R^2-\frac{1}{2}G) \theta\wedge\theta^1\notag\\
 &+(S^1_2+\frac{3}{2}R\tau^1_2+2W^2_2-\frac{1}{2}R_{22})\theta\wedge\theta^2
\notag
  \end{align}
 and conclude that 
   \begin{align}
 Q^1 &=S^1_2+\frac{3}{2}R\tau^1_2+2W^2_2-\frac{1}{2}R_{22}
  \end{align}
\item
Analogously,
substituting the formulas for $\phi^1, \psi$ and $w$, with the above values of $E^1$ and $c$, in the equation $Q^2\theta\wedge \theta^1 	 = d\phi^2 -3\phi^2\wedge w-\theta^2\wedge \psi $
we obtain
 \begin{equation}\label{Q2strict}
 Q^2\theta\wedge \theta^1 =d(-\tau^2-3c\theta^2+ E^2\theta)-3(-\tau^2-3c\theta^2+ E^2\theta)\wedge (\upsilon +c\theta)-\theta^2\wedge  \frac{1}{2}(E^2\theta^1+G\theta).
 \end{equation}
  
  Following the same computations we conclude with the formula
   \begin{align}
 Q^2 &=-S^2_1+\frac{3}{2}R\tau^2_1 -2W^1_1+\frac{1}{2}R_{11}.
  \end{align}

\end{itemize}
 
 \EPf
 
A similar formula is obtained in section of \cite{FVODE} (see proposition  3.5) for the curvature of the path structure adapted bundle in terms of invariants of an enriched structure.   We included a proof here because we used the embedding of the strict path structure adapted bundle  into the path structure adapted bundle instead.   Moreover, conventions are slightly changed.  Namely, in the definition of the strict structure bundle, the following changes should be made from the present paper to \cite{FVODE}:
$\tau^1_2\to \tau^1_2$, $\tau^1_2\to \tau^1_2$, $R\to S$, $W^1\to -C$, $W^2\to -D$, $S^1_2\to \tau^1_{20}$ and $S^2_1\to -\tau^2_{10}$.

\section{}

The authors have no conflict of interest to declare that are relevant to this article.

\section{Data : "Not available".}

\begin{flushleft}
  \textsc{E. Falbel\\
  Institut de Math\'ematiques \\
  de Jussieu-Paris Rive Gauche \\
CNRS UMR 7586 and INRIA EPI-OURAGAN \\
 Sorbonne Universit\'e, Facult\'e des Sciences \\
4, place Jussieu 75252 Paris Cedex 05, France \\}
 \verb|elisha.falbel@imj-prg.fr|
 \end{flushleft}
 
 \begin{flushleft}
  \textsc{M. Mion-Mouton\\
Max Planck Institute for Mathematics in the Sciences\\
in Leipzig\\}
 \verb|martin.mion@mis.mpg.de|
\end{flushleft}
\begin{flushleft}
  \textsc{J. M.  Veloso\\
  Faculdade de Matem\' atica - ICEN\\
Universidade Federal do Par\'a\\66059 - Bel\' em- PA - Brazil}\\
  \verb|veloso@ufpa.br|
\end{flushleft}

\end{document}